\documentclass[reqno,centertags, 11pt]{amsart}

\usepackage{amssymb,amsmath,amsfonts,amssymb}
\usepackage{hyperref}

-\marginparwidth \oddsidemargin -9mm \evensidemargin \oddsidemargin
\usepackage{latexsym}
\advance\hoffset by 5mm

\def\@abssec#1{\vspace{.05in}\footnotesize \parindent .2in
{\bf #1. }\ignorespaces}
\newtheorem{theorem}{Theorem}[section]

\newtheorem{lemma}[theorem]{Lemma}
\newtheorem{proposition}[theorem]{Proposition}
\newtheorem{corollary}[theorem]{Corollary}

\newtheorem{remark}[theorem]{Remark}

\allowdisplaybreaks \numberwithin{equation}{section}

\title[Global regularity for dispersive dissipative QG equation]
{Global regularity for the supercritical dissipative quasi-geostrophic equation with large dispersive forcing}
\author{Marco Cannone}
\thanks{Universit\'e Paris-Est, Laboratorie d'Analyse et de Math\'ematiques Appliqu\'ees, UMR 8050 CNRS, 
5 boulevard Descartes,  Cit\'e Descartes Champs-sur-Marne, 77454 Marne-la-Vall\'ee, Cedex 2, 
France. 
E-mail:marco.cannone@math.univ-mlv.fr}
\author{Changxing Miao}
\thanks{Institute of Applied Physics and Computational Mathematics, P.O. Box 8009,
            Beijing 100088, P.R. China. Email: miao\_{}changxing@iapcm.ac.cn.}
\author{Liutang Xue}
\thanks{The Graduate School of China Academy of Engineering Physics, P.O. Box 2101, Beijing 100088, P.R. China. Email: xue\_{}lt@163.com.}

\begin{document}


\begin{abstract}
  We consider the 2D quasi-geostrophic equation with supercritical dissipation and dispersive forcing in the whole space.
When the dispersive amplitude parameter is large enough, we prove the global well-posedness of strong solution
to the equation with large initial data.
We also show the strong convergence result as the amplitude parameter goes to $\infty$.
Both results rely on the Strichartz-type estimates for the corresponding linear equation.
\end{abstract}


\maketitle

\section{Introduction}

In this paper we consider the following 2D whole-space supercritical dissipative quasi-geostrophic (QG) equation with a dispersive forcing term
\begin{equation}\label{eq 1}
\begin{cases}
  \partial_t \theta + u \cdot\nabla \theta + \nu |D|^\alpha \theta
  + A u_2 =0,   \\
  u =(u_1,u_2)= \mathcal{R}^\perp\theta = (-\mathcal{R}_2\theta,\mathcal{R}_1\theta), \\
  \theta|_{t=0}(x)= \theta_0(x),
\end{cases}
\end{equation}
where $x\in \mathbb{R}^2$, $\alpha\in ]0,1[$, $\nu>0$, $A>0$, $\mathcal{R}_i= -\partial_i |D|^{-1}$ ($i=1,2$) is the usual Riesz transform, and the fractional differential operator $|D|^\alpha$ is defined via the Fourier transform
\begin{equation*}
  \widehat{|D|^\alpha f}(\xi)= |\xi|^{\alpha} \widehat{f}(\xi).
\end{equation*}
Here $\theta$ is a real-valued scalar function that can be interpreted as a buoyancy field,
$A$ is the amplitude parameter. 
This equation \eqref{eq 1} is a simplified model from the geostrophic fluid dynamics and
describes the evolution of a surface buoyancy in the presence of an environmental horizontal buoyancy gradient (cf. \cite{HeldPGS}).
From the physical viewpoint, the background buoyancy gradient generates dispersive waves, and thus the equation \eqref{eq 1}
provides a model for the interaction between waves and turbulent motions in the 2D framework.

When $A=0$, the equation \eqref{eq 1} reduces to the known 2D dissipative quasi-geostrophic equation,
which also arises from the geostrophic fluid dynamics (cf. \cite{HeldPGS,CMT}) and recently has attracted intense attention of many mathematicians
(cf. \cite{CMT,Res,AC-DC,KisNV,CV,Ju05,ChenMZ,Dong2,ConW,Dab} and references therein).
According to the scaling transformation and the $L^\infty$ maximum principle (cf. \cite{AC-DC}), the cases $\alpha>1$, $\alpha=1$ and $\alpha<1$
are referred to as subcritical, critical and supercritical cases respectively.
Up to now, the subcritical and critical cases have been intensively studied. For the delicate critical case,
the issue of global regularity was independently solved by \cite{KisNV} and \cite{CV}.
Kiselev et al in \cite{KisNV} proved the global well-posedness for the periodic smooth data by developing a new method
called the nonlocal maximum principle method. Almost at the same time and from a totally different direction,
Caffarelli and Vasseur in \cite{CV}
established the global regularity of weak solutions by deeply exploiting the DeGiorgi's iteration method.
However, in the supercritical case whether solutions remain globally regular or not is a remarkable open problem.
There are only some partial results, for instance: local well-posedness for large initial data
and global well-posedness for small initial data concerning strong solutions (e.g. \cite{Ju05,ChenMZ,Dong2,HmiK}) and
the eventual regularity of the global weak solutions (e.g. \cite{Dab,Kis}).

The equation \eqref{eq 1} is analogous to the 3D Navier-Stokes equation with Coriolis forcing,
which is a basic model of oceanography and meteorology dealing with large-scale phenomena (cf. \cite{CDGG06}),
\begin{equation}\label{eq NSC}
\begin{cases}
  \partial_t u + u\cdot\nabla u - \nu \Delta u + \nabla P +\frac{1}{\epsilon}e_3\times u =0, \\
  \mathrm{div}\, u =0,\qquad u|_{t=0}= u_0,
\end{cases}
\end{equation}
where $e_3=(0,0,1)$, $\epsilon$ denotes the Rossby number and $u=(u_1,u_2,u_3)$ is the unknown. So far, it is known that
global well-posedness of strong solutions to the 3D Navier-Stokes equations only holds for small initial data. But in the case of the
3D Navier-Stokes-Coriolis system \eqref{eq NSC}, when $\epsilon$ is small enough,
the presence of fast rotating term produces a stabilization effect
and ensures the global well-posedness of strong solution with large initial data.
This result was shown by Babin et al \cite{BMN96,BMN99}
and also by Gallagher \cite{Gal98} in the case of non-resonant periodic domains. For the case of whole space, it was proved
by Chemin et al \cite{CDGG00} through establishing the Strichartz-type estimates of the corresponding linearized system.
By modifying the method of \cite{CDGG00}, Ngo in \cite{Ngo09} moreover studied the case of small viscosity,
i.e., $\nu=\epsilon^\beta$ with $\beta\in ]0,\beta_0]$ and some $\beta_0>0$,
and proved the global existence of strong solution as $\epsilon$ small enough.
The asymptotic behavior of weak solutions in the weak or strong sense as $\epsilon$ goes to $0$ was also considered by
Chemin, Gallagher and their collaborators (cf. \cite{GalSR07,CDGG06}), and they showed that the limiting equation
in the whole space (in general) is the 2D Navier-Stokes equations with the velocity field $\bar u$ three components (not two)
\begin{equation}\label{eq NS2D}
\begin{cases}
  \partial_t \bar u + \bar u_h\cdot\nabla^h\bar u-\nu \Delta_h \bar u + (\nabla^h P,0)=0, \\
  \mathrm{div}\, \bar u_h=0,\qquad \bar u|_{t=0} =\bar u_0,
\end{cases}
\end{equation}
where $\bar u_h\triangleq (\bar u_1,\bar u_2)$, $\Delta_h\triangleq \partial_1^2+\partial_2^2$ and $\nabla^h\triangleq(\partial_1,\partial_2)$.

For the dispersive dissipative QG equation \eqref{eq 1}, Kiselev and Nazarov in \cite{KisN} considered
the critical case of $A>0$ and $\alpha=1$, and by applying the nonlocal maximum principle method they proved the existence of global and regular 
solutions. Note that in their proof, the dispersive term always plays a negative role.

In this paper we mainly focus on the dispersive dissipative QG equation \eqref{eq 1}
in the case of the supercritical regime $\alpha<1$ and $A$ large enough.
Motivated by the results of the 3D Navier-Stokes-Coriolis equations,
we shall develop the Strichartz-type estimate of the corresponding 2D linear equation to prove the global well-posedness of
strong solution to \eqref{eq 1} with large initial data. We shall also show a strong convergence result.

Before stating our main results, we first give some classical uniform existence results.
\begin{proposition}\label{prop uni-exi}
  Let $\theta_0\in L^2(\mathbb{R}^2)$ be a 2D real-valued scalar function. Then there exists a global weak solution $\theta$
(in the sense of distributions) to the dispersive dissipative quasi-geostrophic equation \eqref{eq 1},
which also satisfies the following energy estimate, uniformly in $A$,
\begin{equation*}
  \|\theta(t)\|_{L^2}^2 + 2 \nu \int_0^t
  \||D|^{\frac{\alpha}{2}}\theta(\tau)\|_{L^2}^2\mathrm{d}\tau
  \leq \|\theta_0\|_{L^2}^2, \quad \forall t>0.
\end{equation*}

Moreover if $\theta_0\in H^{2-\alpha}(\mathbb{R}^2)$, then there is a time $T>0$ independent of $A$ such that
$$
\theta\in \mathcal{C}([0,T]; H^{2-\alpha})\cap L^2([0,T]; H^{2-\alpha/2})
$$
with the norm independent of $A$, and all solutions to \eqref{eq 1} coincide with $\theta$ on $[0,T]$. In particular,
an absolute constant $c>0$ can be chosen such that if $\|\theta_0\|_{H^{2-\alpha}}\leq c \nu$, then the solution becomes
global in time.
\end{proposition}

\begin{remark}\label{rem prop1}
  Since for every $s\in\mathbb{R}$, $\overline{|D|^s\theta}=|D|^s\theta$ and
$|\widehat\theta(\xi)|^2=\widehat\theta(\xi)\widehat\theta(-\xi)$, we know that
\begin{equation}\label{eq fact1}
  \int_{\mathbb{R}^2}|D|^{s}\mathcal{R}_1\theta(x)\, |D|^s\theta(x)\mathrm{d}x=\langle|D|^s\mathcal{R}_1\theta,|D|^s\theta\rangle_{L^2}
  =-\int_{\mathbb{R}^2}i\xi_1|\xi|^{2s-1}|\widehat{\theta}(\xi)|^2\mathrm{d}\xi=0,
\end{equation}
thus the dispersive term does not contribute to the energy-type estimates. Therefore the proof of Proposition \ref{prop uni-exi}
is almost identical to the corresponding classical proof for the supercritical dissipative QG equation, and we  omit it here
(cf. \cite{Res,Ju05,ChenMZ,Dong2}).
\end{remark}

Now we consider the asymptotic behavior of the equation \eqref{eq 1} as $A$ tends to infinity. This is reasonable since all bounds
in the above statement are independent of $A$. In what follows we shall also denote by $\theta^A$ the solutions in Proposition
\ref{prop uni-exi} to emphasize the dependence of $A$. The convergence result is as follows.
\begin{theorem}\label{thm conver}
  Let $\theta_0(x)=\bar\theta_0(x_2) + \tilde\theta_0(x)$, with $\bar\theta_0\in H^{\frac{3}{2}-\alpha}(\mathbb{R})$
be a 1D real-valued scalar function and $\tilde\theta_0\in L^2(\mathbb{R}^2)$ be a 2D real-valued scalar function.
Assume that $\bar\theta(t,x_2)$ is the unique solution of the following linear equation
\begin{equation}\label{eq limit}
  \partial_t \bar\theta + \nu|D_2|^\alpha \bar\theta =0 ,\quad \bar\theta(0,x_2)= \bar\theta_0(x_2).
\end{equation}
Then there exists a global weak solution $\theta^A$ to the dispersive dissipative quasi-geostrophic equation \eqref{eq 1}.
Furthermore, for every $\sigma\in ]2,\frac{4}{2-\alpha}[$ and for every $T>0$, we have
\begin{equation}\label{eq conv}
  \lim_{A\rightarrow \infty}\int_0^T \|\theta^A(t) - \bar\theta(t)\|_{L^\sigma}^2\mathrm{d}t=0.
\end{equation}
\end{theorem}

Next we consider the strong solutions, and we prove the following global result.
\begin{theorem}\label{thm stab}
  Let $\theta_0(x)\in H^{2-\alpha}(\mathbb{R}^2)$ be a 2D real-valued scalar
function, then there exists a positive number $A_0$ such that for every $A\geq A_0$,
the dispersive dissipative quasi-geostrophic equation \eqref{eq 1} has a unique global solution $\theta^A$ satisfying
$\theta^A\in \mathcal{C}(\mathbb{R}^+; H^{2-\alpha}(\mathbb{R}^2))\cap L^2(\mathbb{R}^+; \dot H^{2-\alpha/2}(\mathbb{R}^2))$.
Moreover, if we denote by $\tilde\theta^A$ the solution of the following linear dispersive dissipative equation
\begin{equation}\label{eq linDisp}
  \partial_t \tilde\theta^A + \nu |D|^\alpha \tilde\theta^A + A\, \mathcal{R}_1 \tilde\theta^A =0,\quad
  \tilde\theta^A|_{t=0}=\theta_0,
\end{equation}
then as $A$ goes to infinity,
\begin{equation}\label{eq conv2}
  \theta^A -\tilde\theta^A\rightarrow 0, \quad \textrm{in} \quad  L^\infty(\mathbb{R}^+; H^{2-\alpha}(\mathbb{R}^2))
  \cap L^2(\mathbb{R}^+; \dot H^{2-\alpha/2}(\mathbb{R}^2)).
\end{equation}
\end{theorem}

The proofs of both Theorem \ref{thm conver} and Theorem \ref{thm stab} are strongly based on the Strichartz-type estimate
for the corresponding linear equation \eqref{eq linDisp}, which is the target of the whole section \ref{sec Stri}.
The Fourier localization method and the para-differential calculus are also heavily used in the proof of Theorem \ref{thm stab},
and for clarity we place some needed commutator estimates and product estimates in the Appendix section.
The proofs of Theorem \ref{thm conver} and Theorem \ref{thm stab} are settled in Section \ref{sec conv} and Section \ref{sec stab} respectively.

\begin{remark}\label{rem main1}
  The Strichartz-type estimate for the corresponding linear equation depends on the basic dispersive estimate,
which is stated in Lemma \ref{lem disp} below. Compared with the dispersive estimate in the case of the 3D Navier-Stokes-Coriolis equations
(cf. Lemma 5.2 in \cite{CDGG06}), Lemma \ref{lem disp} is much more delicate and the value (precisely, the argument)
of $z$ is more involved in the proof. The main reason is that the equation considered here is two-dimensional,
and the lower dimension makes it harder to to develop the expected dispersive estimate.
This can be further justified if we try to derive the dispersive estimate of the
"anisotropic" kernel function,
i.e., the kernel function as follows
\begin{equation}\label{eq H}
  H(t, \mu,z_2,\xi_1 )\triangleq \int_{\mathbb{R}}\Psi(\xi) e^{i\mu \frac{\xi_1}{|\xi|}+ i z_2 \xi_2 -\nu t |\xi|^\alpha  }\mathrm{d}\xi_2,
\end{equation}
with $z_2\in\mathbb{R}$, $\mu>0$ and $\Psi$ defined by \eqref{eq Psi}, and we find that it is rather difficult
to obtain the needed dispersive estimate.
Notice that the suitable dispersive estimate for \eqref{eq H} will essentially be used if one treats the general data
$\theta_0(x)=\bar\theta_0(x_2)+\tilde\theta_0(x)$ in Theorem \ref{thm stab}.
\end{remark}

\begin{remark}\label{rem main2}
  It is interesting to note that the limiting equation \eqref{eq limit} is analogous to the 2D Navier-Stokes equation \eqref{eq NS2D},
and one can expect that the equation will play a similar role in other situations.
\end{remark}

\section{Preliminaries}\label{sec PRE}
\setcounter{section}{2}\setcounter{equation}{0}

In this preparatory section, we introduce some notations and present the definitions and some
related results of the Sobolev spaces and Besov spaces.

Some notations used in this paper are listed as follows.
\\
$\diamond$ Throughout this paper, $C$ stands for a constant which may be different from line to line. We
sometimes use $A\lesssim B$ instead of $A\leq C B$, and use $A\lesssim_{\beta,\gamma\cdots}B$ instead of
$A\leq C(\beta,\gamma,\cdots)B$, with $C(\beta,\gamma,\cdots)$ a constant depending on $\beta,\gamma,\cdots$.
\\
$\diamond$ Denote by $\mathcal{D}(\mathbb{R}^n)$ the space of test functions which are smooth functions with compact support,
$\mathcal{S}(\mathbb{R}^n)$ the Schwartz space of rapidly
decreasing smooth functions, $\mathcal{S}'(\mathbb{R}^n)$ the
space of tempered distributions,
$\mathcal{S}'(\mathbb{R}^n)/\mathcal{P}(\mathbb{R}^n)$ the
quotient space of tempered distributions up to polynomials.
\\
$\diamond$
$\mathcal{F}f$ or $\widehat{f}$ denotes the Fourier transform, that is
$\mathcal{F}f(\xi)=\widehat{f}(\xi)=\int_{\mathbb{R}^n}e^{-ix\cdot\xi}f(x)\textrm{d} x,$
while $\mathcal{F}^{-1}f$ the inverse Fourier transform, namely,
$\mathcal{F}^{-1}f(x)=(2\pi)^{-n}\int_{\mathbb{R}^n}e^{ix\cdot\xi}f(\xi)\textrm{d}
\xi$ (if there is no ambiguity, we sometimes omit $(2\pi)^{-n}$ for brevity).
\\
$\diamond$ Denote by $\langle f,g\rangle_{L^2}\triangleq \int_{\mathbb{R}^n} f(x) \overline{g}(x)\mathrm{d}x$ the inner product
of the Hilbert space $L^2(\mathbb{R}^n)$.
\\
$\diamond$ Denote by $B(x,r)$ the ball in $\mathbb{R}^n$ centered at $x$ with radius $r$.

Now we give the definition of ($L^2$-based) Sobolev space. For $s\in\mathbb{R}$, the inhomogeneous Sobolev space
\begin{equation*}
  H^{s}\triangleq\Big\{f\in \mathcal{S}'(\mathbb{R}^n);
  \|f\|^2_{H^s}\triangleq \int_{\mathbb{R}^n}(1+|\xi|^2)^s|\widehat{f}(\xi)|^2\textrm{d}
  \xi<\infty\Big\}.
\end{equation*}

Also one can define the corresponding homogeneous space:
\begin{equation*}
  \dot{H}^s\triangleq \Big\{f\in
  \mathcal{S}'(\mathbb{R}^n)/\mathcal{P}(\mathbb{R}^n);
  \|f\|^2_{\dot{H}^s}\triangleq \int_{\mathbb{R}^n}|\xi|^{2s}|\widehat{f}(\xi)|^2\textrm{d} \xi<\infty\Big\}
\end{equation*}

In order to define the Besov spaces, we need the following dyadic partition of unity (cf. \cite{BCD11}).
Choose two nonnegative radial functions $\zeta$, $\psi\in \mathcal{D}(\mathbb{R}^n)$ be
supported respectively in the ball $\{\xi\in \mathbb{R}^n:|\xi|\leq \frac{4}{3} \}$ and the shell $\{\xi\in
\mathbb{R}^n: \frac{3}{4}\leq |\xi|\leq  \frac{8}{3} \}$ such that
\begin{equation*}
  \zeta(\xi)+\sum_{j\geq 0}\psi(2^{-j}\xi)=1, \quad
  \forall\xi\in\mathbb{R}^n; \qquad
  \sum_{j\in \mathbb{Z}}\psi(2^{-j}\xi)=1, \quad \forall\xi\neq 0.
\end{equation*}
For all $f\in\mathcal{S}'(\mathbb{R}^n)$, we define the nonhomogeneous Littlewood-Paley operators
\begin{equation*}
  \Delta_{-1}f \triangleq \zeta(D)f; \qquad
  \Delta_j f \triangleq \psi(2^{-j}D)f,\; S_j f \triangleq\sum_{-1\leq k\leq j-1}\Delta_k f,\quad \forall j\in\mathbb{N},
\end{equation*}
And the homogeneous Littlewood-Paley operators can be defined as follows
\begin{equation*}
  \dot{\Delta}_j f\triangleq \psi(2^{-j}D)f;\quad \dot S_j f\triangleq \sum_{k\in\mathbb{Z},k\leq j-1}\dot \Delta_k f,
  \quad \forall j\in\mathbb{Z}.\qquad
\end{equation*}

Then we introduce the definition of Besov spaces . Let $(p,r)\in[1,\infty]^2$, $s\in\mathbb{R}$, the nonhomogeneous Besov space
\begin{equation*}
  B^s_{p,r}\triangleq\Big\{f\in\mathcal{S}'(\mathbb{R}^n);\|f\|_{B^s_{p,r}}\triangleq\|\{2^{js}\|\Delta
  _j f\|_{L^p}\}_{j\geq -1}\|_{\ell^r}<\infty  \Big\}
\end{equation*}
and the homogeneous space
\begin{equation*}
  \dot{B}^s_{p,r}\triangleq\Big\{f\in\mathcal{S}'(\mathbb{R}^n)/\mathcal{P}(\mathbb{R}^n);
  \|f\|_{\dot{B}^s_{p,r}}\triangleq\|\{2^{js}\|\dot\Delta
  _j f\|_{L^p}\}_{j\in\mathbb{Z}}\|_{\ell^r(\mathbb{Z})}<\infty  \Big\}.
\end{equation*}
We point out that for all $s\in\mathbb{R}$, $B^s_{2,2}=H^s$ and $\dot{B}^s_{2,2}=\dot{H}^s$.

Next we introduce two kinds of space-time Besov spaces. The first one is the classical space-time Besov space
$L^\rho([0,T],B^s_{p,r})$, abbreviated by $L^\rho_{T}B^s_{p,r}$, which is the set of tempered distribution $f$ such that
\begin{equation*}
  \|f\|_{L^\rho_T B^s_{p,r}}\triangleq\|\|\{2^{js}\|\Delta_j f\|_{L^p}\}_{j\geq -1}\|_{\ell^r}\|_{L^\rho([0,T])}<\infty.
\end{equation*}
The second one is the Chemin-Lerner's mixed space-time Besov space $\widetilde{L}^\rho([0,T],B^s_{p,r})$,
abbreviated by $\widetilde{L}^\rho_T B^s_{p,r}$, which is the set of tempered distribution $f$ satisfying
\begin{equation*}
  \|f\|_{\widetilde L^\rho_T B^s_{p,r}}\triangleq \|\{2^{qs}\|\Delta_q f\|_{L^\rho_T L^p}
  \}_{q\geq -1}\|_{\ell^r}<\infty.
\end{equation*}
Due to Minkowiski's inequality, we immediately obtain
\begin{equation*}
 L^\rho_T B^s_{p,r}\hookrightarrow
 \widetilde{L}^\rho_T B^s_{p,r},\;  \mathrm{if}\; r\geq \rho \quad \mathrm{and} \quad
 \widetilde{L}^\rho_T B^s_{p,r}\hookrightarrow L^\rho_T B^s_{p,r},\;  \mathrm{if}\; \rho\geq r.
\end{equation*}
These can similarly extend to the homogeneous one $L^\rho_T\dot B^s_{p,r}$ and $\widetilde L^\rho_T\dot B^s_{p,r}$.

Bernstein's inequality is very fundmental in the analysis involving Besov spaces.
\begin{lemma}\label{lem Bern}
  Let $a, b,p,q$ be positive numbers satisfying $0< a\leq  b< \infty$ and $1\leq p\leq q\leq \infty$, $k\geq 0$, $\lambda>0$ and
$f\in L^p(\mathbb{R}^n)$ with $n\in\mathbb{Z}^+$. Then there exist absolute positive constants $C$ and $c$ such that
\begin{align*}
  \mathrm{if}\;\;\; \mathrm{supp}\widehat{f}\subset \lambda \mathcal{B}\triangleq \{\xi: |\xi|\leq \lambda b\}\Longrightarrow
  \||D|^k f\|_{L^q(\mathbb{R}^n)}\leq C\lambda^{k+n(\frac{1}{p}-\frac{1}{q})}\|f\|_{L^p(\mathbb{R}^n)},
\end{align*}
and
\begin{align*}
  \mathrm{if}\;\; \mathrm{supp}\widehat{f}\subset \lambda \mathcal{C}\triangleq \{\xi: a\lambda\leq |\xi|\leq b\lambda\}
  \Longrightarrow c \lambda^k\|f\|_{L^p(\mathbb{R}^n)}\leq \||D|^k f\|_{L^p(\mathbb{R}^n)}\leq C\lambda^k \|f\|_{L^p(\mathbb{R}^n)}.
\end{align*}
When $k\in\mathbb{N}$, similar estimates hold if $|D|^k$ is replaced by $\sup_{|\gamma|=k}\partial^{\gamma}$.
\end{lemma}

\section{Strichartz-type estimates for the corresponding linear equation}\label{sec Stri}
\setcounter{section}{3}\setcounter{equation}{0}

 This section is devoted to derive the Strichartz-type estimates of the following linear dispersive dissipative
equation
\begin{equation}\label{eq LddQG}
\begin{cases}
  \partial_t\tilde\theta + \nu |D|^\alpha \tilde\theta + A \mathcal{R}_1 \tilde\theta = f, \\
  \tilde\theta|_{t=0}=\tilde\theta_0.
\end{cases}
\end{equation}

Applying the Fourier transformation (in the spatial variable only) to the upper equation, we get
\begin{equation*}
\begin{cases}
  \partial_t \widehat{\tilde\theta} + \nu |\xi|^\alpha \widehat{\tilde\theta} - A i \frac{\xi_1}{|\xi|} \widehat{\tilde\theta}
  = \widehat{f}, \\
  \widehat{\tilde\theta}|_{t=0}=\widehat{\tilde\theta_0}.
\end{cases}
\end{equation*}
Furthermore,
\begin{equation*}
  \widehat{\tilde\theta}(t,\xi) = e^{i A t \frac{\xi_1}{|\xi|} -\nu t |\xi|^\alpha} \widehat{\tilde\theta_0}(\xi)
  + \int_0^t  e^{i A (t-\tau) \frac{\xi_1}{|\xi|} -\nu (t-\tau) |\xi|^\alpha} \widehat{f}(\tau,\xi)\mathrm{d}\tau.
\end{equation*}
Thus by setting
\begin{equation*}
  \mathcal{G}^A(t): g \mapsto \int_{\mathbb{R}^2_\xi} e^{i A t a(\xi)-\nu t|\xi|^\alpha+ix\cdot \xi}
  \widehat{g}(\xi)\mathrm{d}\xi,
\end{equation*}
with
\begin{equation*}
  a(\xi)\triangleq \xi_1/|\xi|,
\end{equation*}
we have
\begin{equation}\label{eq LINEXP}
  \tilde\theta(t)=\mathcal{G}^A(t)\tilde\theta_0 + \int_0^t \mathcal{G}^A(t-\tau) f(\tau)\mathrm{d}\tau.
\end{equation}
Hence, it reduces to consider the Strichartz-type estimate of $\mathcal{G}^A(t)g$,
and because the phase function $a(\xi)$ is somewhat "singular",
we shall study the case when $\widehat{g}$ is supported in the set $\mathcal{B}_{r,R}$ for some $0<r<R$, with
\begin{equation*}
  \mathcal{B}_{r,R}\triangleq \{\xi\in \mathbb{R}^2: |\xi_1|\geq r, \;|\xi|\leq R\}.
\end{equation*}

The main result of this section is as follows.
\begin{proposition}\label{prop Strich}
  Let $r,R$ be two positive numbers satisfying $r<R$, and $g\in L^2(\mathbb{R}^2)$ satisfying
$\mathrm{supp}\;\widehat{g}\subset\mathcal{B}_{r,R} $. Then for every $p\in[1,\infty]$ and $q\in [2,\infty]$,
there exists an absolute constant $C=C_{r,R,p,q,\nu}$ such that
\begin{align}\label{eq Stri}
  \|\mathcal{G}^A(t)g\|_{L^p(\mathbb{R}^+; L^q(\mathbb{R}^2))} \leq C A^{-\frac{1}{8p}(1-\frac{2}{q})}
  \|g\|_{L^2(\mathbb{R}^2)}.
\end{align}
\end{proposition}

Proposition \ref{prop Strich} combined with \eqref{eq LINEXP} and Minkowiski's inequality (cf. \eqref{eq EST2} below)
implies the following Strichartz-type estimates for the linear system \eqref{eq LddQG}.
\begin{corollary}
  Let $r$, $R$ be two positive numbers satisfying $r<R$. Assume that $\tilde\theta_0\in L^2(\mathbb{R}^2)$ and
$f\in L^1(\mathbb{R}^+; L^2(\mathbb{R}^2))$ satisfying that
$$ \mathrm{supp}\, \widehat{\tilde\theta_0}\;\cup\; \Big(\cup_{t\geq 0}\,\mathrm{supp}\,\widehat{f}(t,\cdot)\Big)\subset \mathcal{B}_{r,R}, $$
and $\tilde\theta$ solves the corresponding linear dispersive equation \eqref{eq LddQG}.
Then for every $p\in[1,\infty]$ and $q\in [2,\infty]$, there exists an absolute constant $C=C_{r,R,p,q,\nu}$ such that
\begin{align*}
  \|\tilde\theta\|_{L^p(\mathbb{R}^+; L^q(\mathbb{R}^2))} \leq C A^{-\frac{1}{8p}(1-\frac{2}{q})}\big(\|\tilde\theta_0\|_{L^2(\mathbb{R}^2)}
  + \|f\|_{L^1(\mathbb{R}^+; L^2(\mathbb{R}^2))} \big).
\end{align*}
\end{corollary}

In order to prove Proposition \ref{prop Strich}, we introduce the following kernel function
\begin{equation}\label{eq K}
  K(t, \mu, z )\triangleq \int_{\mathbb{R}^2_\xi}\Psi(\xi) e^{i\mu a(\xi)+ i z\cdot \xi -\nu t |\xi|^\alpha  }\mathrm{d}\xi,
\end{equation}
where $t>0$, $\mu>0$, $z\in\mathbb{R}^2$, $\Psi\in \mathcal{D}(\mathbb{R}^2)$ is a smooth cut-off function such that $\Psi\equiv1$ on
$\mathcal{B}_{r,R}$ and is supported in $\mathcal{B}_{\frac{r}{2},2R}$, and for instance we can explicitly define
\begin{equation}\label{eq Psi}
  \Psi(\xi)=\chi\Big(\frac{|\xi|}{R}\Big)\Big(1-\chi\Big(\frac{2|\xi_1|}{r}\Big)\Big)
\end{equation}
with $\chi\in \mathcal{D}(]-2,2[)$ satisfying $\chi(x)\equiv 1$ for $|x|\leq 1$.

As a first step, we show the following basic dispersive estimate of $K$.
\begin{lemma}\label{lem disp}
  Let $r,R$ be two positive numbers satisfying $r<R$, and $K$ be defined by \eqref{eq K}. Then
there exists an absolute constant $C=C_{r,R}$ such that for every $z\in\mathbb{R}^2$,
\begin{equation*}\label{eq disp}
  |K(t,\mu,z)|\leq C \min\{1,\mu^{-\frac{1}{4}}\} e^{-r^\alpha\nu t/4}.
\end{equation*}
\end{lemma}

\begin{proof}[Proof of Lemma \ref{lem disp}]
We shall use the method of stationary phase to show this formula. Denoting by
$$\Phi(\xi,z)\triangleq \nabla_\xi\Big( a(\xi)+ \frac{z}{\mu}\cdot\xi\Big)= - \frac{\xi_2 }{|\xi|^3}\xi^{\perp}+ \frac{z}{\mu}$$
with $\xi^\perp = (-\xi_2,\xi_1)$, we introduce the following differential operator
\begin{equation*}
  \mathcal{L}\triangleq \frac{\mathrm{Id}-i\Phi(\xi,z)\cdot \nabla_\xi}{1+ \mu |\Phi(\xi,z)|^2},
\end{equation*}
and we see that $\mathcal{L}e^{i \mu a(\xi) + i z \cdot \xi} = e^{i\mu a(\xi)+i z\cdot\xi}$.
From integration by parts, we have
\begin{equation*}
  K(t,\mu, z)= \int_{\mathbb{R}^2}e^{i\mu a(\xi) + iz\cdot\xi} \mathcal{L}^t (\Psi(\xi)e^{-\nu t |\xi|^\alpha})\mathrm{d}\xi,
\end{equation*}
where $\mathcal{L}^t$ is given by
\begin{equation*}
  \mathcal{L}^t \triangleq \frac{1}{1+ \mu |\Phi(\xi,z)|^2}\Big(1 + i \nabla\cdot\Phi-i\frac{2\mu\sum_{j,k}\Phi^j \Phi^k \partial_{\xi_j}
  \Phi^k}{1+\mu|\Phi|^2} \Big)\mathrm{Id}
   + \frac{\Phi(\xi,z)\cdot \nabla_\xi }{1+\mu|\Phi(\xi,z)|^2}.
\end{equation*}
Since $\xi$ is supported in $\mathcal{B}_{\frac{r}{2},2R}$, we find
\begin{equation*}
  |\nabla \Phi(\xi,z)|= \Big|\nabla\Big(\frac{\xi_2 \xi^\perp}{|\xi|^3}\Big) \Big| \lesssim \frac{1}{r^2},
\end{equation*}
thus
\begin{equation*}
  \Big|  1 + i \nabla\cdot\Phi-i\frac{2\mu\sum_{j,k}\Phi^j \Phi^k \partial_{\xi_j}
  \Phi^k}{1+\mu|\Phi|^2}\Big|\lesssim 1+\frac{1}{r^2}.
\end{equation*}
Since $\Psi\in \mathcal{D}(\mathbb{R}^2)$ satisfying $|\nabla \Psi|\lesssim \frac{1}{r}$,
we infer that
\begin{equation*}
\begin{aligned}
  |\nabla(\Psi(\xi) e^{-\nu t|\xi|^\alpha})| & \leq |\nabla \Psi| e^{-\nu t |\xi|^\alpha}
  +  |\Psi|  (\nu t |\xi|^{\alpha}e^{-\nu t|\xi|^\alpha/2})|\xi|^{-1} e^{-\nu t |\xi|^\alpha/2} \\
  & \lesssim \frac{1}{r} e^{-\nu t r^\alpha/4}.
\end{aligned}
\end{equation*}
If $|\Phi(\xi,z)|\geq 1$, this is the case of nonstationary phase, and collecting the upper estimates and
noting that $\mu^{1/2}|\Phi|\leq 1+ \mu|\Phi|^2$, we have
\begin{equation*}
  |\mathcal{L}^t(\Psi(\xi)e^{-\nu t |\xi|^\alpha})|\lesssim \frac{1}{r^2} \min\{1,\mu^{-1/2}\} e^{-\nu t r^{\alpha}/4}.
\end{equation*}
This further yields
\begin{equation*}
  |K(t,\mu, z)|\leq \int_{\mathcal{B}_{r/2,2R}}|\mathcal{L}^t(\Psi(\xi)e^{-\nu t|\xi|^\alpha})|\mathrm{d}\xi
  \lesssim \frac{R^2}{r^2}\min\{1,\mu^{-1/2} \} e^{-\nu t r^\alpha/4}.
\end{equation*}
If $|\Phi(\xi,z)|\leq 1$, this corresponds to the case of stationary case and is more delicate. Gathering the necessary estimates
at above, we have
\begin{equation*}
  |\mathcal{L}^t(\Psi(\xi)e^{-\nu t|\xi|^\alpha})|\lesssim \frac{1}{r^2} \frac{1}{1+ \mu|\Phi(\xi,z)|^2} e^{-\nu t r^\alpha/4},
\end{equation*}
and it further leads to
\begin{equation}\label{eq K1}
  |K(t,\mu, z)|\lesssim \frac{1}{r^2} e^{-\nu t r^\alpha/4} \int_{\mathcal{B}_{r/2,2R}} \frac{1}{1+\mu |\Phi(\xi,z)|^2}\mathrm{d}\xi.
\end{equation}
For the case $z=0$, we see that $\Phi(\xi,0)=-\frac{\xi_2}{|\xi|^3}\xi^\perp$ and $|\Phi(\xi,0)|=\frac{|\xi_2|}{|\xi|^2}$, thus
\begin{equation}
  |K(t,\mu,0)|\lesssim \frac{ R}{r^2} e^{-\nu t r^{\alpha}/4} \int_{0}^{2R}\frac{1}{1+\mu \xi_2^2/(4R^2)}\mathrm{d}\xi_2
  \lesssim \frac{R^2}{r^2} e^{-\nu tr^{\alpha}/4} \mu^{-1/2}.
\end{equation}
For every $\xi\in\mathcal{B}_{\frac{r}{2},2R}$ and $z\in\mathbb{R}^2\setminus\{0\}$, we have the following orthogonal
decomposition
\begin{equation*}
  \xi= (\xi_1, \xi_2) = \xi_{z,\parallel} e_z + \xi_{z,\perp} e_z^\perp,
\end{equation*}
where $e_z=\frac{z}{|z|}$, $e_z^\perp=\frac{z^\perp}{|z|}$,
\begin{equation*}
  \xi_{z,\parallel}\triangleq \xi\cdot e_z,\quad\mathrm{and}\quad \xi_{z,\perp}\triangleq \xi\cdot e_z^\perp.
\end{equation*}
Noting that $\xi^\perp = (-\xi_{z,\perp}) e_z + \xi_{z,\parallel} e_z^\perp$,
we have
\begin{equation*}
  |\Phi(\xi,z)|=\Big|-\frac{\xi_2}{|\xi|^3}\xi^\perp + \frac{z}{\mu}\Big|
  =\Big|\Big(\frac{\xi_2\xi_{z,\perp}}{|\xi|^3} +\frac{|z|}{\mu} \Big)e_z -\frac{\xi_2\xi_{z,\parallel}}{|\xi|^3}e_z^\perp \Big|
  \geq \frac{|\xi_2| | \xi_{z,\parallel}|}{|\xi|^3},
\end{equation*}
and thus for \eqref{eq K1}, it reduces to consider the following integral
\begin{equation*}
  H(\mu, z)\triangleq \int_{\mathcal{B}_{r/2,2R}}\frac{1}{1+ \mu \xi_2^2 \xi_{z,\parallel}^2/(8R^3)}\mathrm{d}\xi.
\end{equation*}
With no loss of generality, we assume that $z=|z|e_z=|z|(\cos\phi,\sin\phi)$ with $\phi\in [0,\frac{\pi}{2}]$.
Then for every $\xi=(\xi_1,\xi_2)$, we have $\xi_{z,\parallel}=\xi\cdot e_z =\xi_1\cos\phi+\xi_2\sin\phi$, thus
if $\xi_1\xi_2\geq 0$, we observe that
\begin{equation*}
  \xi^2_{z,\parallel}\geq \xi_1^2\cos^2\phi+ \xi_2^2\sin^2\phi \geq \min\{\xi_1^2,\xi_2^2\}.
\end{equation*}
Hence
\begin{equation*}
\begin{aligned}
  \int_{\mathcal{B}_{\frac{r}{2},2R}\cap\{\xi_1\xi_2\geq 0\}} \frac{1}{1+ \mu \xi_2^2 \xi_{z,\parallel}^2/(8R^3)}\mathrm{d}\xi_1
  \mathrm{d}\xi_2 & \leq \int_{\mathcal{B}_{\frac{r}{2},2R}\cap\{\xi_1\xi_2\geq 0\}}
  \frac{1}{1+ \mu \xi_2^2\min\{\xi_1^2,\xi_2^2\}/(8 R^3)}\mathrm{d}\xi_1 \mathrm{d}\xi_2 \\
  & \leq  R \int_{-2R}^{2R}\frac{1}{1+ \mu \xi_2^2\min\{r^2/4,\xi_2^2\}/(8 R^3)}\mathrm{d}\xi_2 \\
  & \lesssim \min\Big\{R^2,\max\big\{\frac{R^{5/2}}{r}\mu^{-\frac{1}{2}},R^{7/4}\mu^{-\frac{1}{4}}\big\} \Big\}.
\end{aligned}
\end{equation*}
Now we only need to treat the following integral
\begin{equation*}
\begin{aligned}
  \int_{\mathcal{B}_{\frac{r}{2},2R}\cap \{\xi_1\xi_2\leq 0\}}\frac{1}{1+\mu \xi_2^2 \xi_{z,\parallel}^2/(8R^3)}\mathrm{d}\xi_1
  \mathrm{d}\xi_2  = 2\int_{\mathcal{B}_{\frac{r}{2},2R}\cap \{\xi_1\geq0,\xi_2\leq 0\}}\frac{1}{1+\mu \xi_2^2
  \xi_{z,\parallel}^2/(8R^3)}\mathrm{d}\xi_1\mathrm{d}\xi_2,
\end{aligned}
\end{equation*}
moreover, it suffices to bound the formula from above that for every $\phi\in[0,\pi/2]$,
\begin{equation*}
\begin{aligned}
  \int_{-2R}^0 &  \int_{r/2}^{2R}\frac{1}{1+\mu \xi_2^2
  (\xi_1\cos\phi+\xi_2\sin\phi)^2/(8R^3)}\mathrm{d}\xi_1\mathrm{d}\xi_2
  \\ & = \int_0^{2R} \int_{r/2}^{2R} \frac{1}{1+\mu \xi_2^2 (\xi_1\cos\phi
  - \xi_2\sin\phi)^2/(8R^3)}\mathrm{d}\xi_1\mathrm{d}\xi_2 \\
  & \triangleq \widetilde{H}(\mu,\phi).
\end{aligned}
\end{equation*}
We shall divide into several cases according to $\phi$. First for the endpoint case $\phi=0$, we directly have
\begin{align}
  \widetilde{H}(\mu,0 )& = \int_0^{2R} \int_{r/2}^{2R} \frac{1}{1+\mu \xi_2^2 \xi_1^2/(8R^3)}\mathrm{d}\xi_1\mathrm{d}\xi_2 \nonumber \\
  & \leq 2R \int_0^{2R}\frac{1}{1+\mu \xi_2^2 (r^2/4)/(8R^3)}\mathrm{d}\xi_2 \nonumber \\
  & \lesssim R (\mu r^2/R^3)^{-\frac{1}{2}}\int_0^\infty\frac{1}{1+\tilde\xi_2^2}\mathrm{d}\tilde\xi_2 \lesssim \frac{R^{5/2}}{r}\mu^{-\frac{1}{2}}.
\end{align}
If $\phi$ is close to 0 so that $\frac{r}{2}\cos\phi-2R\sin\phi\geq \frac{r}{4}\cos\phi$, that is, $\phi\in ]0,\arctan(\frac{r}{8R})]$,
we similarly obtain
\begin{equation}
\begin{aligned}
  \widetilde{H}(\mu,\phi) & \leq \int_0^{2R} \int_{r/2}^{2R} \frac{1}{1+\mu \xi_2^2 (r\cos\phi/4)^2/(8R^3)}\mathrm{d}\xi_1\mathrm{d}\xi_2 \\
  & \lesssim \frac{R^{5/2}}{r}\mu^{-\frac{1}{2}}.
\end{aligned}
\end{equation}
For every $\phi\in [\arctan(\frac{r}{8R}),\frac{\pi}{4}]$, if $\xi_2\in [0,r/4]$, we find that $\xi_1-\xi_2 \tan\phi \geq
\frac{r}{2}-\frac{r}{4}\tan(\frac{\pi}{4})=\frac{r}{4}$, thus we get
\begin{equation}
\begin{aligned}
  \widetilde{H}(\mu,\phi)& \leq \int_0^{2R} \int_{r/2}^{2R} \frac{1}{1+\mu \xi_2^2
  (\cos\phi)^2(\xi_1-\xi_2\tan\phi)^2/(8R^3)}\mathrm{d}\xi_1\mathrm{d}\xi_2 \\
  & \lesssim R\int_0^{\frac{r}{4}}\frac{1}{1+\mu \xi_2^2 (\frac{r}{4})^2/(16 R^3)}\mathrm{d}\xi_2
  + \int_{\frac{r}{4}}^{2R} \int_{r/2}^{2R} \frac{1}{1+\mu (\frac{r}{4})^2(\xi_1-\xi_2\tan\phi)^2/(16R)^3 }\mathrm{d}\xi_1\mathrm{d}\xi_2 \\
  & \lesssim \frac{R^{5/2}}{r}\mu^{-\frac{1}{2}} + R \int_{\mathbb{R}}\frac{1}{1+\mu (\frac{r}{4})^2\tilde\xi_1^2/(16R)^3 }\mathrm{d}\tilde\xi_1
  \lesssim \frac{R^{5/2}}{r}\mu^{-\frac{1}{2}}.
\end{aligned}
\end{equation}
For the other endpoint case $\phi=\pi/2$, we directly obtain
\begin{equation}
  \widetilde{H}(\mu,\frac{\pi}{2})\leq \int_0^{2R} \int_{r/2}^{2R}\frac{1}{1+\mu \xi_2^4/(8R^3)} \mathrm{d}\xi_1\mathrm{d}\xi_2
  \lesssim R (\mu/R^3)^{-1/4} \int_{\mathbb{R}}\frac{1}{1+\tilde\xi_2^4}\mathrm{d}\tilde\xi_2\lesssim R^{7/4}\mu^{-\frac{1}{4}}.
\end{equation}
Now we consider the case $\phi\in [\phi_0,\pi/2[$, where $\phi_0\in [\pi/4,\pi/2[$ is a number chosen later. Noticing that
\begin{equation}\label{eq H1}
  \widetilde{H}(\mu,\phi)\leq \int_0^{2R} \int_{r/2}^{2R} \frac{1}{1+\mu \xi_2^2(\xi_1\cot\phi
  - \xi_2)^2/(16R^3)}\mathrm{d}\xi_1\mathrm{d}\xi_2,
\end{equation}
and $\xi_1\geq r/2$, if $\xi_2\leq (\cot\phi) r/4$, we observe that $\xi_1\cot\phi-\xi_2\geq (\cot\phi) (r/4)>0$, thus
\begin{equation*}
\begin{aligned}
  I & \triangleq \int_0^{\frac{r}{4}\cot\phi} \int_{r/2}^{2R}\frac{1}{1+\mu \xi_2^2(\xi_1\cot\phi
  - \xi_2)^2/(16R^3)}\mathrm{d}\xi_1\mathrm{d}\xi_2 \\
  & \leq 2R \int_0^{\frac{r}{4}\cot\phi}\frac{1}{1+\mu \xi_2^2 (\cot\phi)^2(r/4)^2/(16R^3)}\mathrm{d}\xi_2 \\
  & \lesssim \frac{ R^{5/2}}{r\mu^{1/2} \cot\phi} \int_0^{\frac{\mu^{1/2}r^2(\cot\phi)^2}{64R^{3/2}}}
  \frac{1}{1+ \tilde\xi_2^2}\mathrm{d} \tilde\xi_2 \\
  & \lesssim R^{\frac{7}{4}}\mu^{-\frac{1}{4}} \frac{8 R^{3/4}}{\mu^{1/4}r \cot\phi}
  \arctan\Big(\frac{\mu^{1/2}r^2(\cot\phi)^2}{64R^{3/2}}\Big).
\end{aligned}
\end{equation*}
Since $\lim_{x\rightarrow 0+}\frac{\arctan(x^2)}{x}=\lim_{x\rightarrow 0+}\frac{2x}{1+x^2}=0$, there exists an absolute positive
constant $c_0$ such that for every $x\in ]0,c_0]$, we get $\frac{\arctan(x^2)}{x}\leq 1$.
Thus in order to find some $\phi_0\in [\frac{\pi}{4},\frac{\pi}{2}[$ satisfying $\frac{r\mu^{1/4}\cot(\phi_0)}{8R^{3/4}}\leq c_0$,
we only need to choose
\begin{equation*}
  \phi_0\triangleq \max\Big\{\frac{\pi}{4},\arctan\Big(\frac{\mu^{1/4}r}{8c_0 R^{3/4}}\Big)\Big\},
\end{equation*}
then for every $\phi\in [\phi_0,\pi/2[$, we have
\begin{equation*}
  I\lesssim R^{7/4} \mu^{-\frac{1}{4}}.
\end{equation*}
If $\xi_2\geq 4R\cot\phi$, then we find that $|\xi_1\cot\phi-\xi_2|\geq \xi_2-2R\cot\phi\geq \xi_2/2$, thus
\begin{equation*}
\begin{aligned}
  II & \triangleq \int_{4R\cot\phi}^{\infty} \int_{r/2}^{2R}\frac{1}{1+\mu \xi_2^2(\xi_1\cot\phi
  - \xi_2)^2/(16R^3)}\mathrm{d}\xi_1\mathrm{d}\xi_2  \\
  & \leq 2R \int_0^\infty \frac{1}{1+\mu\tilde\xi_2^4/(64R^3)}\mathrm{d} \tilde\xi_2 \lesssim R^{7/4} \mu^{-\frac{1}{4}}.
\end{aligned}
\end{equation*}
If $\xi_2\in [(\cot\phi)r/4,4R\cot\phi]$, noting that $\cot\phi\leq\cot\phi_0 \leq \frac{8c_0R^{3/4}}{r}\mu^{-1/4}$, we have
\begin{equation*}
\begin{aligned}
  III & \triangleq \int^{4R\cot\phi}_{r(\cot\phi)/4} \int_{r/2}^{2R}\frac{1}{1+\mu \xi_2^2(\xi_1\cot\phi
  - \xi_2)^2/(16R^3)}\mathrm{d}\xi_1\mathrm{d}\xi_2  \\
  & \leq (2R) (4R\cot\phi) \lesssim \frac{R^{11/4}}{r} \mu^{-\frac{1}{4}}.
\end{aligned}
\end{equation*}
Hence in the case of $\phi\in [\phi_0,\pi/2[$, we have
\begin{equation}
  \widetilde{H}(\mu,\phi)\leq I+II+III\lesssim \frac{R^{11/4}}{r}\mu^{-\frac{1}{4}}.
\end{equation}
Finally it remains to consider the case $\phi\in [\pi/4,\phi_0]$. Also noticing that \eqref{eq H1} and $\xi_1\geq r/2$,
if $\xi_2\leq r(\cot\phi_0)/4$, we know that $\xi_1\cot\phi-\xi_2\geq (r/2)\cot\phi_0-r(\cot\phi_0)/4=r(\cot\phi_0)/4$,
and combining with the fact that $ r\cot\phi_0=\min\{r , 8c_0R^{3/4}\mu^{-1/4}\}$, we have
\begin{equation*}
\begin{aligned}
  \mathrm{I}_1 &\triangleq \int_0^{r(\cot\phi_0)/4} \int_{r/2}^{2R}\frac{1}{1+\mu \xi_2^2(\xi_1\cot\phi
  - \xi_2)^2/(16R^3)}\mathrm{d}\xi_1\mathrm{d}\xi_2 \\
  & \leq  2R\int_0^\infty \frac{1}{1+\mu\xi_2^2(\min\{r/4,2c_0R^{3/4}\mu^{-1/4}\})^2/(16R^3)}\mathrm{d}\xi_2 \\
  & \lesssim \max\Big\{\frac{R^{5/2}}{r}\mu^{-\frac{1}{2}},R^{7/4} \mu^{-\frac{1}{4}}\Big\}.
\end{aligned}
\end{equation*}
Otherwise, if $\xi_2\geq r(\cot\phi_0)/4= \min\{r/4, 2c_0R^{3/4}\mu^{-1/4}\}$, we infer that
\begin{equation*}
\begin{aligned}
  \mathrm{I}_2 & \triangleq \int_{r(\cot\phi_0)/4}^{2R} \int_{r/2}^{2R}
  \frac{1}{1+\mu \xi_2^2 (\xi_1\cot\phi- \xi_2)^2/(16R^3)}\mathrm{d}\xi_1\mathrm{d}\xi_2 \\
  & \leq R\int_{\mathbb{R}} \frac{1}{1+\mu
  (r(\cot\phi_0)/4)^2 \tilde\xi_{2}^2/(16 R^3)}\mathrm{d}\tilde\xi_{2}\\
  & \lesssim\max\Big\{\frac{R^{5/2}}{r}\mu^{-\frac{1}{2}},R^{7/4} \mu^{-\frac{1}{4}}\Big\}.
\end{aligned}
\end{equation*}
Therefore in the case of $\phi\in[\pi/4,\phi_0]$, we have
\begin{equation}
  \widetilde{H}(\mu,\phi)\leq \mathrm{I}_1+\mathrm{I}_2\lesssim
  \max\Big\{\frac{R^{5/2}}{r}\mu^{-\frac{1}{2}},R^{7/4} \mu^{-\frac{1}{4}}\Big\}.
\end{equation}
This finishes the proof of this Lemma.
\end{proof}

Next we are devoted to proving Proposition \ref{prop Strich} based on Lemma \ref{lem disp}.
\begin{proof}[Proof of Proposition \ref{prop Strich}]
  Noting that
\begin{equation*}
\begin{aligned}
  \mathcal{G}^A(t)g(x)& =\int_{\mathbb{R}^2}\Psi(\xi)\widehat{g}(\xi) e^{iAt a(\xi)-
  \nu t|\xi|^\alpha+ i x\cdot\xi}\mathrm{d}\xi \\
  & = K(t,At,\cdot)\ast g(x),
\end{aligned}
\end{equation*}
where $K$ is defined by \eqref{eq K}, we apply Lemma \ref{lem disp} to obtain
\begin{equation*}
  \|\mathcal{G}^A(t)g\|_{L^{\infty}_x}\lesssim_{r,R}  (At)^{-1/4} e^{-\nu r^\alpha t/4}\|g\|_{L^1}.
\end{equation*}
On the other hand, by the Planchrel theorem, we find
\begin{equation*}
  \|\mathcal{G}^A(t) g\|_{L^2_x}\leq e^{-\nu r^{\alpha} t/2}\|g\|_{L^2}.
\end{equation*}
Thus from interpolation, we have the following dispersive estimates that for every $q\in [2,\infty]$ and $t\in\mathbb{R}^+$,
\begin{equation}\label{eq disp2}
  \|\mathcal{G}^A(t) g\|_{L^q_x} = \| K(t,At,\cdot)\ast g(x)\|_{L^q_x} \lesssim_{r,R} (At)^{-\frac{q-2}{4q}}
  e^{- r^\alpha\nu t/4}\|g\|_{L^{q'}},
\end{equation}
where $q'\triangleq\frac{q}{q-1}$ is the dual number of $q$.

Now we shall use a classical duality method, also called as $TT^*$-method, to show the expected estimates.
For every $q\in [2,\infty]$, denoting by
\begin{equation*}
  \mathcal{U}_q \triangleq \{\varphi\in \mathcal{D}(\mathbb{R}^+\times \mathbb{R}^2):
  \|\varphi\|_{L^\infty(\mathbb{R}^+; L^{q'}(\mathbb{R}^2))}\leq 1\},
\end{equation*}
we have
\begin{equation*}
\begin{aligned}
  \|\mathcal{G}^A(t)g\|_{L^1(\mathbb{R}^+;L^q)} & = \sup_{\varphi\in\mathcal{U}_q }
  \Big|\int_{\mathbb{R}^+}\big\langle\mathcal{G}^A(t)g(x),\varphi(t,x)\big\rangle_{L^2_x}\mathrm{d}t\Big| \\
  & = \sup_{\varphi\in\mathcal{U}_q }\Big|\int_{\mathbb{R}^2} \int_{\mathbb{R}^+}\widehat{g}(\xi) \Psi(\xi)
  e^{iAt a(\xi)-\nu t|\xi|^\alpha}\overline{\widehat{\varphi}}(t,\xi)\mathrm{d}t\mathrm{d}\xi\Big| \\
  & \leq \|g\|_{L^2} \sup_{\varphi\in\mathcal{U}_q } \Big\|
  \int_{\mathbb{R}^+}\Psi(\xi)\overline{\widehat\varphi}(t,\xi)
  e^{iAt a(\xi)-\nu t|\xi|^\alpha}\mathrm{d}t\Big\|_{L^2_\xi}.
\end{aligned}
\end{equation*}
By virtue of the Plancherel theorem, the H\"older inequality and \eqref{eq disp2}, we obtain
\begin{align}
  &\Big\| \int_{\mathbb{R}^+}\Psi(\xi)\overline{\widehat\varphi}(t,\xi) e^{iAt a(\xi)-\nu t|\xi|^\alpha}
  \mathrm{d}t\Big\|^2_{L^2_\xi} \nonumber \\
  =& \int_{\mathbb{R}_\xi^2}\int_{(\mathbb{R}^+)^2}\Psi(\xi)\overline{\widehat{\varphi}}(t,\xi)e^{itA a(\xi)-\nu t|\xi|^\alpha}
  \overline\Psi(\xi) \widehat{\varphi}(\tau,\xi)e^{-i\tau A a(\xi)-\nu \tau |\xi|^\alpha}\mathrm{d}t\,\mathrm{d}\tau\,\mathrm{d}\xi
  \nonumber \\
  =& \int_{\mathbb{R}^+}\int_{\mathbb{R}^+ } \Big\langle \Psi(\xi)\widehat\varphi(\tau,\xi)e^{i(t-\tau ) A
  a(\xi)-\nu(t+\tau)|\xi|^\alpha}, \Psi(\xi)\widehat{\varphi}(t,\xi) \Big\rangle_{L^2_\xi} \,\mathrm{d}\tau\,\mathrm{d}t\nonumber \\
  =& \int_{\mathbb{R}^+}\int_0^t  \Big\langle K(t+\tau,(t-\tau) A,\cdot)\ast
  \varphi(\tau, x),(\mathcal{F}^{-1}\Psi)\ast\varphi(t,x) \Big\rangle_{L^2_x}\,\mathrm{d}\tau\,\mathrm{d}t\, \nonumber \\
  & + \int_{\mathbb{R}^+}\int_t^\infty  \Big\langle(\mathcal{F}^{-1}\Psi)\ast\varphi(\tau,x), K(t+\tau,(\tau-t) A,\cdot)\ast
  \varphi(t, x) \Big\rangle_{L^2_x}\,\mathrm{d}\tau\,\mathrm{d}t\, \nonumber \\
  \leq & \, C_{r,R} \int_{(\mathbb{R}^+)^2} \Big(\frac{1}{A|t-\tau|}\Big)^{\frac{q-2}{4q}}
  e^{-\frac{r^\alpha \nu (t+\tau)}{4}}
  \|\varphi\|^2_{L^\infty(\mathbb{R}^+; L^{q'}(\mathbb{R}^2))} \,\mathrm{d}\tau\,\mathrm{d}t \nonumber.
\end{align}
Since for every $q\in[2,\infty]$, $\varphi\in\mathcal{U}_q $ and
\begin{equation*}
  \int_{(\mathbb{R}^+)^2}\Big(\frac{1}{|t-\tau|}\Big)^{\frac{q-2}{4q}}
  e^{-\frac{r^\alpha \nu (t+\tau)}{4}}\mathrm{d}\tau \mathrm{d}t \leq C_{\nu,q },
\end{equation*}
we get
\begin{equation*}
  \|\mathcal{G}^A(t)g\|_{L^1(\mathbb{R}^+;L^q)}\lesssim_{r,R,q,\nu}
  A^{-\frac{q-2}{8q}} \|g\|_{L^2},\qquad \forall\,q\in [2,\infty].
\end{equation*}
By the Bernstein inequality and the Plancherel theorem, we also have
\begin{equation*}
  \|\mathcal{G}^A(t)g\|_{L^\infty(\mathbb{R}^+;L^q )}\lesssim_R \|\mathcal{G}^A(t)g\|_{L^\infty(\mathbb{R}^+;L^2)}
  \lesssim_{R} \|g\|_{L^2}.
\end{equation*}
From interpolation, we infer that for every $p\in [1,\infty]$ and $q\in [2,\infty]$
\begin{equation*}
  \|\mathcal{G}^A(t)g\|_{L^p(\mathbb{R}^+; L^q )}\lesssim_{r,R,q,p,\nu} A^{-\frac{1}{8p}\frac{q-2}{q}}  \|g\|_{L^2}.
\end{equation*}

\end{proof}

\section{Proof of Theorem \ref{thm conver}}\label{sec conv}
\setcounter{section}{4}\setcounter{equation}{0}

This section is dedicated to the proof of the global existence and convergence of weak solutions to
the dispersive dissipative quasi-geostrophic equation \eqref{eq 1}.

Since $\bar\theta(t,x_2)$ solving \eqref{eq limit} is globally and uniquely defined, we only need to consider the difference $\Theta^A(t,x)\triangleq \theta^A(t,x)-\bar\theta(t,x_2)$, with the associated difference equation formally given by
\begin{equation}\label{eq PddQG}
\begin{cases}
  \partial_t \Theta^A + (\mathcal{R}^\perp\Theta^A)\cdot\nabla\Theta^A - (\mathcal{H}\bar\theta)\,\partial_1\Theta^A +
  \nu|D|^\alpha \Theta^A + A (\mathcal{R}_1 \Theta^A)= - (\mathcal{R}_1\Theta^A)\,\partial_2 \bar\theta, \\
  \Theta^A(0,x)= \tilde\theta_0(x),
\end{cases}
\end{equation}
where $\mathcal{H}$ is the usual Hilbert transform in $\mathbb{R}_{x_2}$.
Note that we have used the following facts that
$\partial_1\bar\theta=0$,
$\mathcal{R}_1\bar\theta=(-|D|^{-1}\partial_1)\bar\theta=0$ and
$$\mathcal{R}_2\bar\theta (x) = \int_{\mathbb{R}^2_\xi} e^{i x\cdot\xi}
\Big(-i\frac{\xi_2}{|\xi|}\Big)\widehat{\bar\theta}(\xi_2)\delta(\xi_1)\mathrm{d}\xi = \int_{\mathbb{R}} e^{ix_2\xi_2}
\Big(-i\frac{\xi_2}{|\xi_2|}\Big)\widehat{\bar\theta}(\xi_2)\mathrm{d}\xi_2 = \mathcal{H}\bar\theta(x_2), $$
and
$$|D|^\alpha\bar\theta(x)=\int_{\mathbb{R}^2}e^{i x\cdot\xi}|\xi|^\alpha\widehat{\bar\theta}(\xi_2)\delta(\xi_1)\mathrm{d}\xi
=\int_{\mathbb R}e^{i x_2\xi_2}|\xi_2|^\alpha\widehat{\bar\theta}(\xi_2)\mathrm{d}\xi_2=|D_2|^\alpha\bar\theta(x_2),$$
with $\delta(\cdot)$ the Dirac-$\delta$ function.

\subsection{Existence of solutions to the perturbed equation \eqref{eq PddQG}}\label{subsec Exi-P}

We first consider the a priori estimates. By taking the $L^2$ inner product of \eqref{eq PddQG} with $\Theta^A$, integration by parts,
and from \eqref{eq fact1} and the fact that $\nabla\cdot (\mathcal{R}^\perp\Theta^A)=0$ and $\partial_1 (\mathcal{H}\bar\theta(x_2))=0$, we get
\begin{align*}
  \frac{1}{2}\frac{d}{dt}\|\Theta^A(t)\|_{L^2}^2 + \nu\big\||D|^{\frac{\alpha}{2}}\Theta^A (t)\big\|_{L^2}^2 =
  -\int_{\mathbb{R}^2}\mathcal{R}_1\Theta^A(t,x)\partial_2\bar\theta(t,x_2)\cdot\Theta^A(t,x)\mathrm{d}x.
\end{align*}
From the H\"older inequality, Sobolev embedding ($\dot H^{\frac{\alpha}{2}}(\mathbb{R})\hookrightarrow L^{\frac{2}{1-\alpha}}(\mathbb{R})$)
and the Calder\'on-Zygmund theorem, we obtain
\begin{align*}
  \frac{1}{2}\frac{d}{dt}\|\Theta^A(t)\|_{L^2}^2 + \nu \big\||D|^{\frac{\alpha}{2}}\Theta^A (t)\big\|_{L^2}^2 & \leq
  \|\mathcal{R}_1\Theta^A(t)\|_{L^{2,2/(1-\alpha)}_{x_1,x_2}}\|\partial_2\bar\theta(t)\|_{L^{2/\alpha}_{x_2}}
  \|\Theta^A(t)\|_{L^2} \\
  & \leq C \big\||D_2|^{\frac{\alpha}{2}}\Theta^A(t)\big\|_{L^2}\|\partial_2\bar\theta(t)\|_{L^{2/\alpha}_{x_2}}\|\Theta^A(t)\|_{L^2} \\
  & \leq C \big\||D|^{\frac{\alpha}{2}}\Theta^A(t)\big\|_{L^2} \|\partial_2\bar\theta(t)\|_{L^{2/\alpha}_{x_2}}\|\Theta^A(t)\|_{L^2}.
\end{align*}
Using the Young inequality, we further have
\begin{align*}
  \frac{1}{2}\frac{d}{dt}\|\Theta^A(t)\|_{L^2}^2 + \frac{\nu}{2} \big\||D|^{\frac{\alpha}{2}}\Theta^A (t)\big\|_{L^2}^2 & \leq
  \frac{ C}{\nu}\|\partial_2\bar\theta(t)\|_{L^{2/\alpha}_{x_2}}^2 \|\Theta^A(t)\|_{L^2}^2.
\end{align*}
Gronwall's inequality ensures that
\begin{align*}
  \|\Theta^A(t)\|_{L^2}^2+ \nu\int_0^t\|\Theta^A(\tau)\|^2_{\dot H^{\frac{\alpha}{2}}}\mathrm{d}\tau
  \leq \|\tilde\theta_0\|^2_{L^2} \exp \Big\{\frac{C}{\nu}\|\partial_2\bar\theta\|_{L^2_t L^{2/\alpha}}^2\Big\}.
\end{align*}
From the Sobolev embedding ($\dot H^{\frac{1-\alpha}{2}}(\mathbb{R})\hookrightarrow L^{\frac{2}{\alpha}}(\mathbb{R})$)
and the energy-type estimate of the linear dissipative equation \eqref{eq limit}, we find
\begin{align}\label{eq fact3}
  \|\partial_2\bar\theta\|^2_{L^2_t L^{2/\alpha}}\lesssim \|\bar\theta\|_{L^2_t \dot H^{3/2-\alpha/2}}^2\lesssim \nu^{-1}
  \|\bar\theta_0\|_{H^{3/2-\alpha}}^2.
\end{align}
Hence, we finally obtain that for every $t\in\mathbb{R}^+$
\begin{align}\label{eq EgyE}
  \|\Theta^A(t)\|_{L^2}^2+ \nu\int_0^t\|\Theta^A(\tau)\|^2_{\dot H^{\frac{\alpha}{2}}}\mathrm{d}\tau
  \leq \|\tilde\theta_0\|^2_{L^2}\exp\Big\{\frac{C}{\nu^2}\|\bar\theta_0\|_{H^{3/2-\alpha}}^2\Big\} .
\end{align}

Next we sketch the proof of the global existence of solution to \eqref{eq PddQG}. We have the following approximate system
\begin{equation}\label{eq app1}
\begin{cases}
  \partial_t \Theta_\epsilon^A + (\mathcal{R}^\perp\Theta_\epsilon^A)\cdot\nabla\Theta_\epsilon^A -
  (\mathcal{H}\bar\theta_\epsilon)\,\partial_1\Theta_\epsilon^A +
  \nu|D|^\alpha \Theta_\epsilon^A + A (\mathcal{R}_1 \Theta_\epsilon^A)- \epsilon \Delta \Theta^A_\epsilon
  = - (\mathcal{R}_1\Theta_\epsilon^A)\,\partial_2 \bar\theta_\epsilon, \\
  \partial_t \bar \theta_\epsilon + \nu|D_2|^\alpha \bar \theta_\epsilon =0, \\
  \Theta^A(0,x)= \varphi_\epsilon\ast\tilde\theta_0(x),\quad \bar\theta_\epsilon(0,x_2) =\tilde\varphi_\epsilon\ast\bar\theta_0(x_2),
\end{cases}
\end{equation}
where $\varphi_\epsilon(x) = \epsilon^{-2}\varphi(x/\epsilon)$ and $\varphi\in \mathcal{D}(\mathbb{R}^2)$ satisfies
$\int_{\mathbb{R}^2}\varphi=1$, $\tilde\varphi_\epsilon(x_2)=\epsilon^{-1}\tilde\varphi(x_2/\epsilon)$ and
$\tilde\varphi\in \mathcal{D}(\mathbb{R})$ satisfies
$\int_{\mathbb{R}}\tilde\varphi=1$. Let $m>2$ and $m\in\mathbb{Z}^+$, and fix $\epsilon>0$. Since
$\|\varphi_\epsilon\ast\tilde\theta_0\|_{H^m}\lesssim_\epsilon \|\tilde\theta_0\|_{L^2}$ and
$\|\tilde\varphi_\epsilon\ast\bar\theta_0\|_{H^m_{x_2}}\lesssim_\epsilon \|\bar\theta_0\|_{L_{x_2}^2}$,
and since $-\epsilon\Delta\Theta^A_\epsilon$
is the subcritical dissipation, from the standard energy method we find that for all $T>0$
\begin{align*}
  \sup_{t\in[0,T]}\|\Theta^A_\epsilon(t)\|_{H^m} \leq C(\epsilon,T,\varphi,\tilde\varphi,\|\tilde\theta_0\|_{L^2}, \|\bar\theta_0\|_{L^2}).
\end{align*}
This estimate combined with a Galerkin approximation process yields the global existence of a strong solution
$(\Theta^A_\epsilon,\bar\theta_\epsilon)$ to \eqref{eq app1}. Furthermore, from \eqref{eq EgyE} and the estimation
$\|\varphi_\epsilon\ast f\|_{H^s}\leq \|f\|_{H^s}$, $\forall s\in\mathbb{R}$, we have the uniform energy inequality
with respect to $\epsilon$ that for all $T>0$
\begin{equation}\label{eq EgyE1}
\begin{aligned}
  \|\Theta_\epsilon^A(T)\|_{L^2}^2+ \nu\int_0^T\|\Theta_\epsilon^A(s)\|^2_{\dot H^{\frac{\alpha}{2}}}\mathrm{d}s
  & \leq \|\varphi_\epsilon\ast\tilde\theta_0\|^2_{L^2}\exp\Big\{\frac{C}{\nu^2}\|\tilde\varphi_\epsilon\ast\bar\theta_0\|_{H^{3/2-\alpha}}^2\Big\} \\
  & \leq \|\tilde\theta_0\|^2_{L^2}\exp\Big\{\frac{C}{\nu^2}\|\bar\theta_0\|_{H^{3/2-\alpha}}^2\Big\}.
\end{aligned}
\end{equation}
Hence this ensures that, up to a subsequence, $\Theta^A_\epsilon$ converges weakly (or weakly-$*$) to a function $\Theta^A$ in
$L^{\infty}_T L^2\cap L^2_T \dot H^{\alpha/2}$. Similarly as the case of the dissipative quasi-geostrophic equation,
from the compactness argument, we further get that as $\epsilon$ tends to 0,
\begin{equation*}
\begin{cases}
  \Theta^A_\epsilon\rightarrow \Theta^A, \\
  \mathcal{R}_j \Theta^A_\epsilon \rightarrow \mathcal{R}_j\Theta^A, j=1,2,
\end{cases}
\mathrm{strongly\;in\;} L^2([0,T]; L^2_{\mathrm{loc}}(\mathbb{R}^2)).
\end{equation*}
Since $\bar\theta_0\in H^{3/2-\alpha}(\mathbb{R})$, it is clear to see that $\bar\theta_\epsilon$ strongly converges to $\bar\theta=e^{-\nu t|D_2|}\bar\theta_0$
in $L^\infty\big([0,T]; H^{\frac{3}{2}-\alpha}(\mathbb{R})\big)$. Therefore we can pass to the limit in \eqref{eq app1} to show that
$\Theta^A$ is a weak solution of \eqref{eq PddQG}.

\subsection{Proof of \eqref{eq conv}.}
Now we show the strong convergence of $\Theta^A$ by using the Strichartz-type estimate \eqref{eq Stri}.
To this end, we introduce the following cutoff operator
\begin{align}\label{eq I_rR}
  \mathcal{I}_{r,R}=\mathcal{I}_{r,R}(D)\triangleq \chi\big(|D|/R\big)\big(\mathrm{Id}-\chi(|D_1|/r)\big),
\end{align}
where $0<r<R$ and $\chi\in \mathcal{D}(\mathbb{R})$ satisfies that
$\chi(x)\equiv 1$ for all $|x|\leq 1$ and $\chi$ is compactly supported in $\{x : |x|< 2\}$.
Then for the term $\mathcal{I}_{r,R}\Theta^A$, we have the following estimation (with its proof placed in the end of this subsection).
\begin{lemma}\label{lem StrApp}
  Let $r,R$ be two positive numbers satisfying $r<R$. Then for every $T>0$ and $\sigma\in ]2,\infty[$,
there exists an absolute constant $\widetilde{C}$ depending on $r,R,T,\sigma,\nu$, $\|\bar\theta_0\|_{H^{\frac{3}{2}-\alpha}}$ and $\|\tilde\theta_0\|_{L^2}$
but independent of $A$ such that
\begin{align}\label{eq StrApp}
  \|\mathcal{I}_{r,R}\Theta^A\|_{L^2([0,T]; L^\sigma(\mathbb{R}^2))}\leq \widetilde{C} A^{-\frac{1}{16}(1-\frac{2}{\sigma})}.
\end{align}
\end{lemma}
%

Now we consider the contribution from the part of high frequency and the part of low frequency in $\xi_1$.
From the Sobolev embedding, Berenstein inequality
and the energy estimate \eqref{eq EgyE}, we get for every $\sigma\in [2,\frac{4}{2-\alpha}[$,
\begin{equation*}
\begin{aligned}
  \|(\mathrm{Id}-\chi(|D|/R))\Theta^A\|_{L^2(\mathbb{R}^+;L^\sigma(\mathbb{R}^2))}
  & \lesssim \|(\mathrm{Id}-\chi(|D|/R))\Theta^A\|_{L^2(\mathbb{R}^+;\dot H^{1-2/\sigma}(\mathbb{R}^2))} \\
  & \lesssim R^{-(\frac{\alpha}{2}+\frac{2}{\sigma}-1)}\|(\mathrm{Id}-\chi(|D|/R))\Theta^A\|_{L^2(\mathbb{R}^+;
  \dot H^{\frac{\alpha}{2}}(\mathbb{R}^2))} \\
  & \lesssim R^{-(\frac{2}{\sigma}-\frac{2-\alpha}{2})} \|\tilde\theta_0\|_{L^2}
  \exp\Big\{\frac{C}{\nu^2}\|\bar\theta_0\|_{H^{3/2-\alpha}}^2\Big\} .
\end{aligned}
\end{equation*}
Also thanks to the Bernstein inequality (in $x_1$ and $x_2$ separately), we find that for every $T>0$ and $\sigma\in ]2,\infty]$,
\begin{align*}
  \|\chi(|D_1|/r)\chi(|D|/R)\Theta^A\|_{L^2([0,T];L^\sigma(\mathbb{R}^2))}& \lesssim T^{\frac{1}{2}} r^{(\frac{1}{2}-\frac{1}{\sigma})}
  R^{\frac{1}{2}-\frac{1}{\sigma}} \|\Theta^A\|_{L^\infty([0,T]; L^2(\mathbb{R}^2))} \\
  & \lesssim_{T,R} r^{\frac{1}{2}-\frac{1}{\sigma}} \|\tilde\theta_0\|_{L^2}\exp\Big\{\frac{C}{\nu^2}\|\bar\theta_0\|_{H^{3/2-\alpha}}^2\Big\} .
\end{align*}
Collecting the upper estimates, we have that for every $A,r,R,T>0$ and $\sigma\in ]2,\frac{4}{2-\alpha}[$,
\begin{align*}
  \|\Theta^A\|_{L^2([0,T]; L^\sigma(\mathbb{R}^2))}\leq C_0  R^{-(\frac{2}{\sigma}-\frac{2-\alpha}{2})}
  +C_{R,T} r^{\frac{1}{2}-\frac{1}{\sigma}} + \widetilde{C} A^{-\frac{1}{16}(1-\frac{2}{\sigma})},
\end{align*}
where $\widetilde{C}$ depends on $r,R,T$, $\|\tilde\theta_0\|_{L^2}$ and $\|\bar\theta_0\|_{H^{3/2-\alpha}}$ but not on $A$.
Hence, passing $A$ to $\infty$, then $r$ to $0$ and then $R$ to $\infty$ yields the desired estimate \eqref{eq conv}.

At last it suffices to prove Lemma \ref{lem StrApp}.
\begin{proof}[Proof of Lemma \ref{lem StrApp}]
  By virtue of Duhamel's formula, we have
\begin{align*}
  \mathcal{I}_{r,R}\Theta^A =\, & \mathcal{G}^A(t) \mathcal{I}_{r,R}\tilde\theta_0 -
  \int_0^t\mathcal{G}^A(t-\tau)\;\mathcal{I}_{r,R}\big(\mathcal{R}^\perp\Theta^A\cdot\nabla \Theta^A\big)(\tau)\mathrm{d}\tau \\
  & +\int_0^t\mathcal{G}^A(t-\tau)\;\mathcal{I}_{r,R}\big(\mathcal{H}\bar\theta\, \partial_1\Theta^A
  -\mathcal{R}_1\Theta^A\partial_2\bar\theta\big)(\tau)\mathrm{d}\tau \\
  \triangleq \,& \Gamma_1 -\Gamma_2 +\Gamma_3.
\end{align*}
From the Strichartz-type estimate \eqref{eq Stri}, we know that for every $\sigma\in ]2,\infty[$
\begin{align*}
  \|\Gamma_1\|_{L^2(\mathbb{R}^+; L^\sigma(\mathbb{R}^2))}\lesssim_{\sigma,r,R} A^{-\frac{1}{16}(1-\frac{2}{\sigma})}
  \|\tilde\theta_0\|_{L^2(\mathbb{R}^2)}.
\end{align*}
Applying the Minkowski inequality and again \eqref{eq Stri} to $\Gamma_2$, we infer that for every $\sigma\in ]2,\infty[$
and $T>0$,
\begin{equation}\label{eq EST2}
\begin{aligned}
  \|\Gamma_2\|_{L^2([0,T]; L^\sigma(\mathbb{R}^2))}& \leq \Big(\int_0^T \Big|\int_0^T
  1_{[0,t]}(\tau)\|\mathcal{G}^A(t-\tau)\,\mathcal{I}_{r,R}\big(
  \mathcal{R}^\perp\Theta^A\cdot\nabla\Theta^A\big)(\tau)\|_{L^\sigma} \mathrm{d}\tau\Big|^2\mathrm{d}t\Big)^{1/2} \\
  & \leq \int_0^T \Big( \int_\tau^T
  \|\mathcal{G}^A(t-\tau)\,\mathcal{I}_{r,R}\big(
  \mathcal{R}^\perp\Theta^A\cdot\nabla\Theta^A\big)(\tau)\|_{L^\sigma}^2 \mathrm{d}t\Big)^{1/2}\mathrm{d}\tau \\
  & \lesssim_{r,R,\sigma} A^{-\frac{1}{16}(1-\frac{2}{\sigma})} \int_0^T \|\mathcal{I}_{r,R}\big(
  \mathcal{R}^\perp\Theta^A\cdot\nabla\Theta^A\big)(\tau)\|_{L^2}\mathrm{d}\tau.
\end{aligned}
\end{equation}
From Bernstein's inequality and the energy estimate \eqref{eq EgyE}, we further get
\begin{equation*}
\begin{aligned}
  \|\mathcal{I}_{r,R}(\mathcal{R}^\perp\Theta^A\cdot\nabla\Theta^A)\|_{L^1([0,T]; L^2(\mathbb{R}^2))}
  &\lesssim R^{2} \|(\mathcal{R}^\perp\Theta^A)\Theta^A\|_{L^1([0,T]; L^1(\mathbb{R}^2))} \\
  & \lesssim R^2 T \|\Theta^A\|_{L^\infty([0,T]; L^2(\mathbb{R}^2))}^2 \\
  & \lesssim R^2 T \|\tilde\theta_0\|^2_{L^2} \exp\Big\{\frac{C}{\nu^2}\|\bar\theta_0\|_{H^{3/2-\alpha}}^2\Big\}.
\end{aligned}
\end{equation*}
Thus we have
\begin{align*}
  \|\Gamma_2\|_{L^2([0,T]; L^\sigma(\mathbb{R}^2))}\lesssim_{r,R,\sigma} A^{-\frac{1}{16}(1-\frac{2}{\sigma})}T
  \|\tilde\theta_0\|^2_{L^2} \exp\Big\{\frac{C}{\nu^2}\|\bar\theta_0\|_{H^{3/2-\alpha}}^2\Big\}.
\end{align*}
For $\Gamma_3$, similarly as above, especially from Bernstein's inequality in $x_2$-variable and the Calder\'on-Zygmund theorem,
we infer that for every $\sigma\in ]2,\infty[$ and $T>0$,
\begin{align*}
  \|\Gamma_3\|_{L^2([0,T]; L^\sigma(\mathbb{R}^2))}& \lesssim A^{-\frac{1}{16}(1-\frac{2}{\sigma})}
  \|\mathcal{I}_{r,R}\big((\mathcal{H}\bar\theta)\partial_1\Theta^A
  -(\mathcal{R}_1\Theta^A)\partial_2\bar\theta\big)\|_{L^1([0,T];L^2(\mathbb{R}^2))} \\
  & \lesssim A^{-\frac{1}{16}(1-\frac{2}{\sigma})}\Big(R^{\frac{3}{2}}\|(\mathcal{H}\bar\theta)\Theta^A\|_{L^1([0,T]; L^{2,1}_{x_1,x_2})}
  + R^{\frac{3}{2}}\|(\mathcal{R}_1\Theta^A)\bar\theta\|_{L^1([0,T]; L^{2,1}_{x_1,x_2})}\Big) \\
  & \lesssim A^{-\frac{1}{16}(1-\frac{2}{\sigma})}R^{\frac{3}{2}}T \|\bar\theta\|_{L^\infty([0,T]; L^2(\mathbb{R}))}
  \|\Theta^A\|_{L^\infty([0,T]; L^2(\mathbb{R}^2))} \\
  & \lesssim A^{-\frac{1}{16}(1-\frac{2}{\sigma})} R^{\frac{3}{2}}T \|\bar\theta_0\|_{L^2}\|\tilde\theta_0\|_{L^2}
  \exp\Big\{\frac{C}{\nu^2}\|\bar\theta_0\|_{H^{3/2-\alpha}}^2 \Big\}.
\end{align*}
Hence, gathering the upper estimates leads to the expected estimate \eqref{eq StrApp}.

\end{proof}

\section{Proof of Theorem \ref{thm stab}}\label{sec stab}
\setcounter{section}{5}\setcounter{equation}{0}

Now we show the global existence of $\theta^A$ as stated in Theorem \ref{thm stab}.
If we only consider the equation \eqref{eq 1} to get the $H^{2-\alpha}$ estimates of $\theta^A$, due to the estimation \eqref{eq fact1},
it seems impossible to derive an estimate global in time unless the data $\theta_0$ are small enough
(just as Proposition \ref{prop uni-exi}). Thus we shall adopt an idea from
\cite{CDGG06}, that is, to subtract from the equation \eqref{eq 1} the solution $\tilde\theta^A$ of the linear equation
\eqref{eq linDisp} (or its main part $\mathcal{I}_{r,R}\tilde\theta^A$ with $\mathcal{I}_{r,R}$ defined in \eqref{eq I_rR}).
Roughly speaking, since from the Strichartz-type estimate \eqref{eq Stri}, $\tilde\theta^A$ can be sufficiently small
for $A$ large enough, thus the equation of $\theta^A-\tilde\theta^A$ will have small initial data and small forcing terms,
and we can hope to get the global existence result.

More precisely, we first introduce $\tilde\theta^A_m\triangleq \mathcal{I}_{r,R}\tilde\theta^A$ as the main part of $\tilde\theta^A$
which solves the following equation
\begin{equation}\label{eq theA-m}
  \partial_t \tilde\theta_m^A + \nu |D|^\alpha \tilde\theta_m^A + A\, \mathcal{R}_1 \tilde\theta_m^A =0,\quad
  \tilde\theta_m^A|_{t=0}=\mathcal{I}_{r,R}\theta_0,
\end{equation}
and since $\mathcal{I}_{r,R}\theta_0$ strongly converges to $\theta_0$ in $H^{2-\alpha}(\mathbb{R}^2)$
as $r$ tends to $0$ and $R$ tends to $\infty$, the difference $\tilde\theta^A-\tilde\theta^A_m$
is globally defined and can be made arbitrarily small in the functional spaces stated in Theorem \ref{thm stab}.
Hence, in the sequel we shall focus on the difference $\eta^A\triangleq\theta^A-\tilde\theta^A_m$ with $r$ small enough and $R$ large enough
chosen later, and we shall be devoted to show the global existence of $\eta^A$. The corresponding equation can be written as
\begin{equation}\label{eq etaA}
\begin{aligned}
  \partial_t \eta^A  + \nu |D|^\alpha\eta^A +A\mathcal{R}_1\eta^A + (\mathcal{R}^\perp\eta^A) \cdot\nabla\eta^A
  +(\mathcal{R}^\perp\tilde\theta^A_m)\cdot\nabla\eta^A &= F(\eta^A,\tilde\theta_m^A), \\
    \eta^A|_{t=0}& =(\mathrm{Id}-\mathcal{I}_{r,R})\theta_0.
\end{aligned}
\end{equation}
with the forcing term
\begin{align}\label{eq force}
  F(\eta^A,\tilde\theta_m^A)\triangleq
  -(\mathcal{R}^\perp\tilde\theta^A_m)\cdot\nabla\tilde\theta^A_m-(\mathcal{R}^\perp\eta^A )\cdot\nabla\tilde\theta^A_m.
\end{align}
Note that for brevity, we have omitted the dependence of $r,R$ in the notation of $\eta^A$ and $\tilde\theta^A_m$.

\subsection{A priori estimates.}\label{sec apri}

In this subsection, we mainly care about the a priori estimates. The main result is the following claim: for any smooth solution $\eta^A$ to
\eqref{eq etaA} and for every $\epsilon>0$ small enough, there exist three positive absolute constants $r_0,R_0,A_0$
such that for every $A\geq A_0$, we have
\begin{align}\label{eq claim}
  \sup_{t\geq 0}\|\eta^A(t)\|_{H^{2-\alpha}(\mathbb{R}^2)}^2 + \frac{\nu}{2} \int_{\mathbb{R}^+}\|\eta^A(t)\|_{\dot H^{2-\frac{\alpha}{2}}
  (\mathbb{R}^2)}^2\mathrm{d}t \leq \epsilon.
\end{align}
For every $q\in \mathbb{N}$, applying $\Delta_q$ to the equation \eqref{eq etaA} and denoting $\eta^A_q\triangleq \Delta_q \eta^A$,
$F_q\triangleq \Delta_q F$, we get
\begin{align*}
  \partial_t\eta^A_q +\nu|D|^\alpha\eta^A_q +A(\mathcal{R}_1\eta^A_q) + (\mathcal{R}^\perp\eta^A)\cdot\nabla\eta^A_q +
  (\mathcal{R}^\perp\tilde\theta^A_m)\cdot\nabla\eta^A_q = \widetilde{F}_q(\eta^A,\tilde\theta^A_m),
\end{align*}
with
\begin{align*}
  \widetilde{F}_q(\eta^A,\tilde\theta^A_m)\triangleq
   - [\Delta_q,\mathcal{R}^\perp\eta^A]\cdot\nabla\eta^A
  - [\Delta_q,\mathcal{R}^\perp\tilde\theta^A_m]\cdot\nabla\eta^A+ F_q(\eta^A,\tilde\theta^A_m).
\end{align*}
Since $\eta^A$ is real-valued, we know that $\eta^A_q$ is also real-valued, thus taking $L^2$ inner product of the upper equation with
$\eta^A_q$, and from the Bernstein inequality and the integration by parts, we obtain
\begin{align*}
  \frac{1}{2}\frac{d}{dt} \|\eta^A_q(t)\|^2_{L^2}+\nu \||D|^{\frac{\alpha}{2}}\eta^A_q(t)\|_{L^2}^2 &  = \int_{\mathbb{R}^2}
  \widetilde{F}_q(\eta^A,\tilde\theta^A_m)\,\eta^A_q(t,x)\mathrm{d}x \\
  & \leq 2^{-q\frac{\alpha}{2}}\|\widetilde{F}_q(\eta^A,\tilde\theta^A_m)(t)\|_{L^2}\, 2^{q\frac{\alpha}{2}} \|\eta^A_q(t)\|_{L^2} \\
  & \leq C_0 2^{-q\frac{\alpha}{2}}\|\widetilde{F}_q(\eta^A,\tilde\theta^A_m)(t)\|_{ L^2}\, \||D|^{\frac{\alpha}{2}}\eta^A_q(t)\|_{L^2}.
\end{align*}
From Young's inequality, we have
\begin{align*}
  \frac{1}{2}\frac{d}{dt} \|\eta^A_q(t)\|^2_{L^2}+\frac{\nu}{2} \||D|^{\frac{\alpha}{2}}\eta^A_q(t)\|_{L^2}^2
  \leq \frac{C_0}{\nu}\Big( 2^{-q\frac{\alpha}{2}}\|\widetilde{F}_q(\eta^A,\tilde\theta^A_m)(t)\|_{L^2}\Big)^2.
\end{align*}
Integrating in time leads to
\begin{align*}
  \|\eta^A_q(t)\|_{L^2}^2+ \nu \||D|^{\frac{\alpha}{2}}\eta^A_q\|_{L^2_t L^2}^2 \leq
  \|\eta^A_q(0)\|_{L^2}^2+
  \frac{C_0}{\nu} \int_0^t  2^{-q\alpha}\|\widetilde{F}_q(\eta^A,\tilde\theta^A_m)(\tau)\|^2_{L^2}\mathrm{d}\tau.
\end{align*}
By multiplying both sides of the upper inequality by $2^{2q(2-\alpha)}$ and summing over all $q\in\mathbb{N}$, we get
\begin{equation}
\begin{aligned}\label{eq HF}
  & \sum_{q\in\mathbb{N}}2^{2q (2-\alpha)}\|\eta^A_q(t)\|^2_{L^2}+ \nu \sum_{q\in\mathbb{N}}2^{2q(2-\alpha)}
  \||D|^{\frac{\alpha}{2}}\eta^A_q\|_{L^2_t L^2}^2  \\
  \leq & \sum_{q\in\mathbb{N}}2^{2q (2-\alpha)}\|\eta^A_q(0)\|^2_{L^2}+ \frac{C_0}{\nu}  \int_0^t
  \Big(\sum_{q\in\mathbb{N}}2^{2q(2-\frac{3\alpha}{2})}\|\widetilde{F}_q(\eta^A,\tilde\theta^A_m)(\tau)\|^2_{L^2}\Big)\mathrm{d}\tau.
\end{aligned}
\end{equation}
Using Lemma \ref{lem comm} with $s=2-\alpha$ and $\beta=\frac{\alpha}{2}$ yields that
\begin{align*}
  \sum_{q\in\mathbb{N}} 2^{2q(2-\frac{3}{2}\alpha)}\|[\Delta_q, \mathcal{R}^\perp (\eta^A + \tilde\theta^A_m)]\cdot\nabla \eta^A\|_{L^2}^2
  & \lesssim_\alpha \big( \|\eta^A\|^2_{\dot B^{2-\frac{\alpha}{2}}_{2,2}}+ \|\tilde\theta^A_m\|^2_{\dot B^{2-\frac{\alpha}{2}}_{2,2}}\big)
  \|\eta^A\|^2_{B^{2-\alpha}_{2,2}} \\
  & \lesssim_\alpha \big( \||D|^{\frac{\alpha}{2}}\eta^A\|^2_{B^{2-\alpha}_{2,2}}+
  \||D|^{\frac{\alpha}{2}}\tilde\theta^A_m\|^2_{B^{2-\alpha}_{2,2}}\big)
  \|\eta^A\|^2_{B^{2-\alpha}_{2,2}},
\end{align*}
where in the second line we have used the embedding $B^{2-\alpha}_{2,2}\hookrightarrow \dot B^{2-\alpha}_{2,2}$.
From Lemma \ref{lem product}-(1), we infer that
\begin{align*}
  \sum_{q\in\mathbb{N}} 2^{2q(2-\frac{3}{2}\alpha)}\|\Delta_q(\mathcal{R}^\perp\eta^A \cdot\nabla \tilde\theta^A_m)\|_{L^2}^2
  \lesssim R^{6-3\alpha} \|\eta^A\|_{L^2}^2 \|\tilde\theta^A_m\|_{L^\infty}^2 +
  \|\eta^A\|^2_{B^{2-\alpha}_{2,2}} \||D|^{\frac{\alpha}{2}}\tilde\theta^A_m\|^2_{B^{2-\alpha}_{2,2}},
\end{align*}
and
\begin{align*}
  \sum_{q\in\mathbb{N}} 2^{2q(2-\frac{3}{2}\alpha)} \|\Delta_q(\mathcal{R}^\perp\tilde\theta_m^A \cdot\nabla \tilde\theta^A_m)\|_{L^2}^2
  \lesssim R^{6-3\alpha} \|\tilde\theta_m^A\|_{L^2}^2 \|\tilde\theta^A_m\|_{L^\infty}^2.
\end{align*}
Inserting the upper estimates to \eqref{eq HF}, we obtain
\begin{equation}\label{eq HF1}
\begin{aligned}
  & \sum_{q\in\mathbb{N}}2^{2q (2-\alpha)}\|\eta^A_q(t)\|^2_{L^2}+ \nu \sum_{q\in\mathbb{N}}2^{2q(2-\alpha)}
  \||D|^{\frac{\alpha}{2}}\eta^A_q\|_{L^2_t L^2}^2 \\
  \leq & \sum_{q\in\mathbb{N}}2^{2q (2-\alpha)}\|\eta^A_q(0)\|^2_{L^2}
  +  \frac{C_\alpha}{\nu} \int_0^t \big( \||D|^{\frac{\alpha}{2}}\eta^A(\tau)\|^2_{B^{2-\alpha}_{2,2}}+
  \||D|^{\frac{\alpha}{2}}\tilde\theta^A_m(\tau)\|^2_{B^{2-\alpha}_{2,2}}\big)
  \|\eta^A(\tau)\|^2_{B^{2-\alpha}_{2,2}}\mathrm{d}\tau  \\ & +  \frac{C_\alpha}{\nu} R^{6-3\alpha}  \int_0^t \big(\|\tilde\theta_m^A(\tau)\|_{L^2}^2
   + \|\eta^A(\tau)\|_{L^2}^2 \big)  \|\tilde\theta^A_m(\tau)\|_{L^\infty}^2\mathrm{d}\tau.
\end{aligned}
\end{equation}
Now we consider the low-frequency part. Applying $\Delta_{-1}$ to the equation \eqref{eq etaA}, we have
\begin{align*}
  \partial_{t}(\Delta_{-1}\eta^A) + \nu |D|^\alpha (\Delta_{-1}\eta^A) + A \mathcal{R}_1 (\Delta_{-1}\eta^A)= \Delta_{-1} G(\eta^A,\tilde\theta^A_m),
\end{align*}
where
\begin{align*}
  G(\eta^A,\tilde\theta^A_m)\triangleq -(\mathcal{R}^\perp\eta^A\cdot\nabla\eta^A)
  -(\mathcal{R}^\perp\tilde\theta^A_m\cdot\nabla\eta^A)
  -(\mathcal{R}^\perp\eta^A\cdot\nabla\tilde\theta^A_m)
  -(\mathcal{R}^\perp\tilde\theta^A_m\cdot\nabla\tilde\theta^A_m).
\end{align*}
By using the $L^2$ energy method, we obtain
\begin{align*}
  \frac{1}{2}\frac{d}{dt}\|\Delta_{-1}\eta^A(t)\|^2_{L^2} + \nu \||D|^{\frac{\alpha}{2}} \Delta_{-1}\eta^A(t)\|_{L^2}^2
  & \leq  \|\Delta_{-1}G(\eta^A,\tilde\theta^A_m)(t)\|_{\dot H^{-\frac{\alpha}{2}}} \||D|^{\frac{\alpha}{2}}\Delta_{-1}\eta^A(t)\|_{L^2}.
\end{align*}
By virtue of the Young inequality, we have
\begin{align*}
  \frac{1}{2}\frac{d}{dt}\|\Delta_{-1}\eta^A(t)\|^2_{L^2} + \frac{\nu}{2} \||D|^{\frac{\alpha}{2}} \Delta_{-1}\eta^A(t)\|_{L^2}^2
  & \leq  \frac{C_0}{\nu} \|\Delta_{-1}G(\eta^A,\tilde\theta^A_m)(t)\|_{\dot H^{-\frac{\alpha}{2}}}^2.
\end{align*}
Integrating in time yields that
\begin{align}\label{eq LF}
  \|\Delta_{-1}\eta^A(t)\|^2_{L^2} + \nu \||D|^{\frac{\alpha}{2}} \Delta_{-1}\eta^A\|_{L^2_t L^2}^2
  & \leq \|\Delta_{-1}\eta^A(0)\|^2_{L^2} +
  \frac{C_0}{\nu} \int_0^t \|\Delta_{-1}G(\eta^A,\tilde\theta^A_m)(\tau)\|_{\dot H^{-\frac{\alpha}{2}}}^2\mathrm{d}\tau.
\end{align}
From Lemma \ref{lem product}-(2), we deduce that
\begin{align*}
  \|\Delta_{-1}(\mathcal{R}^\perp\eta^A\cdot\nabla\eta^A)\|^2_{\dot H^{-\frac{\alpha}{2}}}
  & \lesssim \sum_{-\infty< q\leq 0}2^{-q\alpha}
  \|\dot\Delta_q(\mathcal{R}^\perp\eta^A\cdot\nabla\eta^A)\|_{L^2}^2 \\
  & \lesssim \sum_{-\infty<q\leq 0}2^{2q(2-\alpha)} \||D|^{\frac{\alpha}{2}}\eta^A\|^2_{L^2} \|\eta^A\|^2_{L^2} \\
  & \lesssim  \||D|^{\frac{\alpha}{2}}\eta^A\|^2_{L^2} \|\eta^A\|^2_{L^2},
\end{align*}
and
\begin{align*}
  \|\Delta_{-1}(\mathcal{R}^\perp\tilde\theta^A_m\cdot\nabla\eta^A)\|^2_{\dot H^{-\frac{\alpha}{2}}}+
  \|\Delta_{-1}(\mathcal{R}^\perp\eta^A\cdot\nabla\tilde\theta^A_m)\|^2_{\dot H^{-\frac{\alpha}{2}}}
  \lesssim  \||D|^{\frac{\alpha}{2}} \tilde\theta^A_m\|_{L^2}^2 \|\eta^A \|^2_{L^2} .
\end{align*}
It is also obvious to see that
\begin{align*}
  \|\Delta_{-1}(\mathcal{R}^\perp\tilde\theta^A_m\cdot\nabla\tilde\theta^A_m )\|^2_{\dot H^{-\frac{\alpha}{2}}} \lesssim
  \|\Delta_{-1}((\mathcal{R}^\perp\tilde\theta^A_m) \,\tilde\theta^A_m )\|^2_{\dot H^{1-\frac{\alpha}{2}}} \lesssim \|\tilde\theta^A_m\|_{L^2}^2 \|\tilde\theta^A_m\|_{L^\infty}^2.
\end{align*}
Inserting the upper estimates into \eqref{eq LF}, we obtain
\begin{equation}\label{eq LF1}
\begin{aligned}
  &  \|\Delta_{-1}\eta^A(t)\|^2_{L^2} + \nu \||D|^{\frac{\alpha}{2}} \Delta_{-1}\eta^A\|_{L^2_t L^2}^2 \\
  \leq &  \|\Delta_{-1}\eta^A(0)\|^2_{L^2} +  \frac{C_\alpha}{\nu} \int_0^t \big(\||D|^{\frac{\alpha}{2}}\eta^A(\tau)\|^2_{L^2}
  +\||D|^{\frac{\alpha}{2}} \tilde\theta^A_m(\tau)\|_{L^2}^2\big) \|\eta^A(\tau)\|^2_{L^2}\mathrm{d}\tau \\
  & + \frac{C_\alpha}{\nu} \int_0^t\|\tilde\theta^A_m(\tau)\|_{L^2}^2 \|\tilde\theta^A_m(\tau)\|_{L^\infty}^2 \mathrm{d}\tau.
\end{aligned}
\end{equation}
Combining this estimate with \eqref{eq HF1} leads to
\begin{equation*}
\begin{aligned}
  &  \|\eta^A(t)\|_{B^{2-\alpha}_{2,2}}^2 + \nu \||D|^{\frac{\alpha}{2}}\eta^A\|_{L^2_t B^{2-\alpha}_{2,2}}^2 \\
  \leq & \|\eta^A(0)\|_{B^{2-\alpha}_{2,2}}^2 + \frac{C_\alpha}{\nu} \int_0^t \big( \||D|^{\frac{\alpha}{2}}\eta^A(\tau)\|^2_{B^{2-\alpha}_{2,2}}+
  \||D|^{\frac{\alpha}{2}}\tilde\theta^A_m(\tau)\|^2_{B^{2-\alpha}_{2,2}}\big)
  \|\eta^A(\tau)\|^2_{B^{2-\alpha}_{2,2}}\mathrm{d}\tau  \\ & + \frac{C_\alpha}{\nu} R^{6-3\alpha} \int_0^t \big(\|\tilde\theta_m^A(\tau)\|_{L^2}^2
   + \|\eta^A(\tau)\|_{L^2}^2 \big) \|\tilde\theta^A_m(\tau)\|_{L^\infty}^2\mathrm{d}\tau.
\end{aligned}
\end{equation*}
From the fact that $\|\eta^A\|_{L^\infty_t L^2}\leq \|\tilde\theta^A_m\|_{L^\infty_t L^2} + \|\theta^A\|_{L^\infty_t L^2}\leq 2\|\theta_0\|_{L^2}$,
we moreover find that
\begin{equation}\label{eq EST1}
\begin{aligned}
  &  \|\eta^A\|_{L^\infty_t B^{2-\alpha}_{2,2}}^2 + \nu \||D|^{\frac{\alpha}{2}}\eta^A\|_{L^2_t B^{2-\alpha}_{2,2}}^2 \\
  \leq & \|\eta^A(0)\|_{B^{2-\alpha}_{2,2}}^2 + \frac{C_\alpha}{\nu} \int_0^t
  \||D|^{\frac{\alpha}{2}}\tilde\theta^A_m(\tau)\|^2_{B^{2-\alpha}_{2,2}}
  \|\eta^A\|^2_{L^\infty_\tau B^{2-\alpha}_{2,2}}\mathrm{d}\tau  \\ &
  + \frac{C_\alpha}{\nu}\|\eta^A\|_{L^\infty_t B^{2-\alpha}_{2,2}}^2  \||D|^{\frac{\alpha}{2}}\eta^A\|_{L^2_t B^{2-\alpha}_{2,2}}^2
  + \frac{C_\alpha}{\nu} R^{6-3\alpha}\|\theta_0\|_{L^2}^2 \int_0^t   \|\tilde\theta^A_m(\tau)\|_{L^\infty}^2\mathrm{d}\tau.
\end{aligned}
\end{equation}
Set
\begin{align*}
  T_A^* \triangleq \mathrm{sup}\Big\{t\geq 0;  \; \|\eta^A\|^2_{L^\infty_t B^{2-\alpha}_{2,2}}< \frac{\nu^2}{2C_\alpha}\Big\},
\end{align*}
due to $\|\eta^A(0)\|^2_{B^{2-\alpha}_{2,2}}=\|(\mathrm{Id}-\mathcal{I}_{r,R})\theta_0\|_{B^{2-\alpha}_{2,2}}^2 $, and by the Lebesgue theorem,
we can choose some small number $r$ and large number $R$ such that $\|\eta^A(0)\|_{B^{2-\alpha}_{2,2}}^2\leq \frac{\nu^2}{4C_\alpha}$,
thus $T^*_A>0$ follows from that $\eta^A$ is a (continuous in time) smooth solution.
Then, through the Strichartz-type estimate \eqref{eq Stri}, we obtain that for every $t\in [0,T^*_A[$,
\begin{equation*}
\begin{aligned}
  \|\eta^A\|_{L^\infty_t B^{2-\alpha}_{2,2}}^2 + \frac{\nu}{2} \||D|^{\frac{\alpha}{2}}\eta^A\|_{L^2_t B^{2-\alpha}_{2,2}}^2
  \leq &\, \|\eta^A(0)\|_{B^{2-\alpha}_{2,2}}^2   + \frac{C_\alpha}{\nu} R^{6-3\alpha} C_{r,R} A^{-\frac{1}{8}}\|\theta_0\|^4_{L^2}
  \\ & + \frac{C_\alpha}{\nu} \int_0^t
  \||D|^{\frac{\alpha}{2}}\tilde\theta^A_m(\tau)\|^2_{B^{2-\alpha}_{2,2}}
  \|\eta^A\|^2_{L^\infty_\tau B^{2-\alpha}_{2,2}}\mathrm{d}\tau .
\end{aligned}
\end{equation*}
Gronwall's inequality yields that for every $t\in [0,T^*_A[$
\begin{equation*}
\begin{aligned}
  \|\eta^A\|_{L^\infty_t B^{2-\alpha}_{2,2}}^2 + \frac{\nu}{2} \||D|^{\frac{\alpha}{2}}\eta^A\|_{L^2_t B^{2-\alpha}_{2,2}}^2
  \leq &\, \exp\Big\{\frac{C_\alpha}{\nu^2} \|\theta_0\|^2_{H^{2-\alpha}}\Big\}
  \Big(\|\eta^A(0)\|_{B^{2-\alpha}_{2,2}}^2   + \frac{C_{r,R,\alpha}}{\nu} A^{-\frac{1}{8}}\|\theta_0\|^4_{L^2}\Big),
\end{aligned}
\end{equation*}
where we have used the following fact that
\begin{align*}
  \|\tilde\theta^A_m(t)\|^2_{B^{2-\alpha}_{2,2}} + \nu \||D|^{\frac{\alpha}{2}}\tilde\theta^A_m\|^2_{L^2_t B^{2-\alpha}_{2,2}}\leq
  \| \theta_0\|_{B^{2-\alpha}_{2,2}}^2
  \leq C_0 \|\theta_0\|^2_{H^{2-\alpha}}.
\end{align*}
For every $\epsilon>0$, we can further choose some small number $r$ and large number $R$ such that
\begin{equation*}
  \|(\mathrm{Id}-\mathcal{I}_{r,R})\theta_0\|^2_{B^{2-\alpha}_{2,2}} \exp\Big\{\frac{C_\alpha}{\nu^2} \|\theta_0\|^2_{H^{2-\alpha}}\Big\}
  \leq \frac{\epsilon}{2C_0},
\end{equation*}
where $C_0$ is the absolute constant from the relation $\frac{1}{C_0}\|f\|^2_{B^s_{2,2}}\leq \|f\|^2_{H^s} \leq C_0\|f\|^2_{B^s_{2,2}}$,
$\forall s\in\mathbb{R}$.
For fixed $r,R$, we can choose $A$ large enough so that
\begin{align*}
   A^{-\frac{1}{8}}\frac{C_{r,R,\alpha}}{\nu}  \|\theta_0\|^4_{L^2}
   \exp\Big\{\frac{C_\alpha}{\nu^2} \|\theta_0\|^2_{H^{2-\alpha}}\Big\}
   \leq  \frac{\epsilon}{2C_0}.
\end{align*}
Hence for every $\epsilon>0$ and for the appropriate $r,R,A$ (i.e. $r_0,R_0,A\geq A_0$), we have
\begin{align*}
  \sup_{t\in[0,T^*_A[}\|\eta^A(t)\|_{ B^{2-\alpha}_{2,2}}^2 + \frac{\nu}{2} \int_0^{T_A^*}
  \||D|^{\frac{\alpha}{2}}\eta^A(t)\|_{ B^{2-\alpha}_{2,2}}^2\mathrm{d}t
  \leq \frac{\epsilon}{C_0}.
\end{align*}
Furthermore, for every $\epsilon\leq \frac{C_0 \nu^2}{4C_\alpha}$, we have $T^*_A=\infty$ and
\begin{align*}
  \sup_{t\in\mathbb{R}^+}\|\eta^A(t)\|_{ B^{2-\alpha}_{2,2}}^2 + \frac{\nu}{2} \int_0^\infty
  \||D|^{\frac{\alpha}{2}}\eta^A(t)\|_{ B^{2-\alpha}_{2,2}}^2\mathrm{d}t
  \leq \frac{\epsilon}{C_0}.
\end{align*}
Therefore \eqref{eq claim} follows.

\subsection{Uniqueness.}
For every $T>0$, let $\theta^A_1$ and $\theta^A_2$ belonging to
$$L^\infty([0,T]; H^{2-\alpha}(\mathbb{R}^2))\cap L^2([0,T];\dot H^{2-\frac{\alpha}{2}}(\mathbb{R}^2))$$
be two solutions to \eqref{eq 1} with the same initial data $\theta_0\in H^{2-\alpha/2}(\mathbb{R}^2)$.
Thus set $\delta\theta^A\triangleq \theta^A_1-\theta^A_2$, and then the difference equation writes
\begin{equation*}
\begin{split}
  \partial_t\delta\theta^A + (\mathcal{R}^\perp\theta^A_1)\cdot\nabla\delta\theta^A
  +\nu |D|^\alpha \delta\theta^A +A (\mathcal{R}_1\delta\theta^A) & = -(\mathcal{R}^\perp\delta\theta^A)\cdot\nabla\theta^A_2 , \\
  \delta\theta^A|_{t=0}&=\delta\theta_0^A (=0).
\end{split}
\end{equation*}
We use the $L^2$ energy argument to get
\begin{align*}
  \frac{1}{2}\frac{d}{dt}\|\delta\theta^A(t)\|_{L^2}^2 +\frac{\nu}{2} \big\||D|^{\frac{\alpha}{2}}\delta\theta^A(t)\big\|_{L^2}^2
  \leq&  \|(\mathcal{R}^\perp\delta\theta^A)\cdot\nabla\theta^A_2(t) \|_{\dot H^{-\frac{\alpha}{2}}}
  \big\||D|^{\frac{\alpha}{2}}\delta\theta^A(t)\big\|_{L^2}
\end{align*}
From the following classical product estimate that for every divergence-free vector field $f\in \dot H^{s_1}(\mathbb{R}^2)$ and
$g\in \dot H^{s_2}(\mathbb{R}^2)$
with $s_1,s_2<1$ and $s_1+s_2>-1$,
\begin{align*}
  \|f\cdot\nabla g\|_{\dot H^{s_1+s_2-1}(\mathbb{R}^2)}\lesssim_{s_1,s_2} \|f\|_{\dot H^{s_1}(\mathbb{R}^2)} \|\nabla g\|_{H^{s_2}(\mathbb{R}^2)},
\end{align*}
we know that
\begin{align*}
  \|(\mathcal{R}^\perp\delta\theta^A)\cdot\nabla\theta^A_2(t) \|_{\dot H^{-\frac{\alpha}{2}}}\lesssim_\alpha
  \|\delta\theta^A\|_{L^2} \|\nabla\theta^A_2\|_{\dot H^{1-\frac{\alpha}{2}}}.
\end{align*}
Thanks to the Young inequality, we further find
\begin{align}\label{eq uni1}
  \frac{d}{dt}\|\delta\theta^A\|_{L^2}^2 +\nu \|\delta\theta^A\|_{\dot H^{\frac{\alpha}{2}}}^2 \lesssim_{\alpha,\nu}
    \|\nabla\theta^A_2(t)\|^2_{\dot H^{1-\frac{\alpha}{2}}} \|\delta\theta^A(t)\|_{L^2}^2.
\end{align}
Gronwall's inequality yields
\begin{align*}
  \|\delta\theta^A(t)\|_{L^2}^2 \leq \|\delta\theta_0^A\|_{L^2}^2 \exp\Big\{
  C \|\theta^A_2\|^2_{L^2([0,T]; \dot H^{2-\frac{\alpha}{2}}(\mathbb{R}^2))} \Big \}\lesssim
  \|\delta\theta_0\|_{L^2}^2 \exp\Big\{\frac{C}{\nu^2} \|\theta_0\|_{H^{2-\alpha}}^2 \Big\}.
\end{align*}
Hence the uniqueness is guaranteed.

\subsection{Global existence.}
From the Friedrich method, we consider the following approximate system
\begin{equation}\label{eq appExi}
\begin{cases}
  \partial_t \eta_k^A  + \nu |D|^\alpha\eta_k^A +A\mathcal{R}_1\eta_k^A + J_k\big(\mathcal{R}^\perp\eta_k^A \cdot\nabla\eta_k^A\big)
  +J_k\big(\mathcal{R}^\perp\tilde\theta^A_m\cdot\nabla\eta_k^A\big) = J_k F(\eta_k^A,\tilde\theta_m^A), \\
    \eta_k^A|_{t=0} =J_k(\mathrm{Id}-\mathcal{I}_{r,R})\theta_0,
\end{cases}
\end{equation}
where $J_k: L^2\mapsto J_k L^2$, $k\in\mathbb{N}$ is the projection operator such that
$J_k f \triangleq \mathcal{F}^{-1}(1_{B(0,k)}(\xi)\widehat f(\xi))$ and $\tilde\theta^A_m$ solving \eqref{eq theA-m}
is the main part of $\tilde\theta^A$. Indeed the system \eqref{eq appExi} becomes an
ordinary differential equation on the space $J_k L^2\triangleq \{f\in L^2: \mathrm{supp}\,\widehat f\subset B(0,k)\}$
with the $L^2$ norm. Since
$$\|J_k(\mathcal{R}^\bot \eta_k^A\cdot\nabla \eta_k^A)\|_{L^2} \lesssim k \|\mathcal{R}^\perp\eta^A_k\|_{L^2}\|\nabla\eta^A_k\|_{L^2}
\lesssim k^2 \|\eta^A_k\|_{L^2}^2,$$
and
$$\|J_k\big(\mathcal{R}^\perp\tilde\theta^A_m\cdot\nabla\eta^A_k+ \mathcal{R}^\perp\eta^A_k \cdot\nabla\tilde\theta^A_m\big) \|_{L^2}
\lesssim k^2 \|\theta_0\|_{L^2} \|\eta^A_k\|_{L^2}, $$
and $\|J_k(\mathcal{R}^\perp\tilde\theta^A_m\cdot\nabla \tilde\theta^A_m)\|_{L^2} \lesssim k^2 \|\theta_0\|_{L^2}^2 $,
we have that for every $r,R>0$ and $k\in \mathbb{N}$, there exists a unique solution $\eta^A_k\in \mathcal{C}^\infty([0,T_k[;J_k L^2)$
to the system \eqref{eq appExi}, with $T_k>0$ the maximal existence time. Moreover, from the $L^2$ energy method and in a similar way as obtaining
\eqref{eq uni1}, we get
\begin{align*}
  \frac{d}{dt}\|\eta^A_k\|_{L^2}^2 + \nu \|\eta^A_k\|^2_{\dot H^{\frac{\alpha}{2}}} \lesssim_{\nu} \|\eta^A_k\|_{L^2}^2
  \|\tilde\theta^A_m\|_{\dot H^{2-\frac{\alpha}{2}}}^2  + \|\tilde\theta^A_m\|_{L^2}^2 \|\tilde\theta^A_m\|_{\dot H^{2-\frac{\alpha}{2}}}^2.
\end{align*}
Gronwall's inequality and the energy-type estimate of the linear equation \eqref{eq linDisp} yield that
\begin{align*}
  \|\eta^A_k(t)\|_{L^2}^2 & \leq \exp\{C_\nu \|\tilde\theta^A_m\|_{L^2_t \dot H^{2-\frac{\alpha}{2}}} \}
  \big( \|\eta^A_k(0)\|_{L^2}^2 + \|\tilde\theta^A_m \|^2_{L^\infty_t L^2} \|\tilde\theta^A_m\|^2_{L^2_t \dot H^{2-\frac{\alpha}{2}}}\big) \\
  & \leq \exp\{C_\nu \|\theta_0\|_{H^{2-\alpha}}\} \big(\|\theta_0\|_{L^2}^2 + \|\theta_0\|_{L^2}^2 \|\theta_0\|_{H^{2-\alpha}}^2 \big).
\end{align*}
Hence the classical continuation criterion ensures that $T_k=\infty$ and $\eta^A_k\in \mathcal{C}^\infty(\mathbb{R}^+; J_k L^2)$ is a global solution
to the system \eqref{eq appExi}. This further guarantees the a priori estimate in Section \ref{sec apri}, that is, we obtain that
there exist positive absolute constants $\epsilon_0,r_0,R_0,A_0$ independent of $k$ such that for every $0<\epsilon<\epsilon_0$ and $A>A_0$,
\begin{align*}
  \sup_{t\in\mathbb{R}^+}\|\eta^A_k(t)\|_{H^{2-\alpha}}^2 + \nu \int_{\mathbb{R}^+}\|\eta^A_k(t)\|^2_{\dot H^{2-\frac{\alpha}{2}}}\mathrm{d}t
  \leq \epsilon.
\end{align*}
Based on this uniform estimate, it is not hard to show that $(\eta^A_k)_{k\in\mathbb{N}}$ is a Cauchy sequence in
$\mathcal{C}(\mathbb{R}^+; L^2(\mathbb{R}^2))$, and thus it converges strongly to a function $\eta^A\in\mathcal{C}(\mathbb{R}^+; L^2(\mathbb{R}^2))$.
By a standard process, one can prove that $\eta^A$ solves the system \eqref{eq etaA} and
$\eta^A\in L^\infty(\mathbb{R}^+; H^{2-\alpha}(\mathbb{R}^2))\cap L^2(\mathbb{R}^+; \dot H^{2-\frac{\alpha}{2}}(\mathbb{R}^2))$.
Moreover, from the proof in Section \ref{sec apri} and by replacing $\|\eta^A\|_{L^\infty_t B^{2-\alpha}_{2,2}}$ with
$\|\eta^A\|_{\widetilde{L}^\infty_t B^{2-\alpha}_{2,2}}$ in \eqref{eq EST1}, one indeed can prove that
$\eta^A\in\widetilde{L}^\infty(\mathbb{R}^+; B^{2-\alpha}_{2,2}(\mathbb{R}^2))$,
and this implies that $\eta^A\in \mathcal{C}(\mathbb{R}^+; H^{2-\alpha}(\mathbb{R}^2))$. Finally, let $\theta^A=\eta^A+\tilde\theta^A_m$, then
for $A$ large enough $\theta^A$ be the unique solution to the dispersive dissipative QG equation \eqref{eq 1},
and as $A\rightarrow\infty$, $r\rightarrow 0$, $R\rightarrow \infty$ and $\epsilon\rightarrow 0$ one-by-one,
we get the expected convergence \eqref{eq conv2}.

\section{Appendix}
\setcounter{section}{6}\setcounter{equation}{0}

We first consider some commutator estimates.
\begin{lemma}\label{lem comm}
  Let $v=(v_1,\cdots,v_n)$ be a smooth divergence-free vector field over $\mathbb{R}^n$ and $f$ be a smooth scalar function of $\mathbb{R}^n$.
Then for every $q\in\mathbb{N}$, $\beta\in ]0,1+\frac{n}{2}[$ and $s\in ]\beta-1-\frac{n}{2},1+\frac{n}{2}[$, there exists a positive absolute constant $C$ depending
only on $\beta,s,n$ such that
\begin{align*}
  2^{q(s-\beta)}\|[\Delta_q,v]\cdot\nabla f\|_{L^2(\mathbb{R}^n)} \leq C c_q \| v\|_{\dot B_{2,2}^{1+\frac{n}{2}-\beta}(\mathbb{R}^n)}
  \|f\|_{B_{2,2}^s(\mathbb{R}^n)},
\end{align*}
where $(c_q)_{q\in\mathbb{N}}$ satisfies $\sum_{q\in\mathbb{N}}(c_q)^2\leq 1$. In particular, if $n=2$ and
$v=\mathcal{R}^\perp f$, we also have that for every $\beta>0$ and $s>\beta-1-\frac{n}{2}$,
\begin{align*}
  2^{q(s-\beta)}\|[\Delta_q,v]\cdot\nabla f\|_{L^2(\mathbb{R}^n)} \leq C c_q \| f\|_{\dot B_{2,2}^{1+\frac{n}{2}-\beta}(\mathbb{R}^n)}
  \|f\|_{ B_{2,2}^s(\mathbb{R}^n)},
\end{align*}
with $(c_q)_{q\in\mathbb{N}}$ satisfying $\sum_{q\in\mathbb{N}}(c_q)^2\leq 1$.
\end{lemma}

\begin{proof}[Proof of Lemma \ref{lem comm}]
  From Bony's decomposition we have
\begin{equation*}
\begin{split}
  [\Delta_q,v]\cdot\nabla f & = \sum_{|k-q|\leq 4}[\Delta_q,S_{k-1}v]\cdot\nabla\Delta_k f
  +\sum_{|k-q|\leq 4} [\Delta_q,\Delta_k v]\cdot\nabla S_{k-1}f
  + \sum_{k\geq q-3}[\Delta_q, \Delta_k v]\cdot\nabla\widetilde{\Delta}_k f \\
  & \triangleq \mathrm{I}_q +\mathrm{II}_q + \mathrm{III}_q.
\end{split}
\end{equation*}
For $\mathrm{I}_q$, by virtue of the expression $\Delta_q=h_q(\cdot)\ast=2^{qn}h(2^q\cdot)\ast$ with $h\triangleq
\mathcal{F}^{-1}(\psi)\in\mathcal{S}(\mathbb{R}^n)$, we get
\begin{equation*}
\begin{aligned}
  2^{q(s-\beta)}\|\mathrm{I}_q\|_{L^2} & \lesssim 2^{q(s-\beta)} \sum_{|k-q|\leq 4}\|xh_q\|_{L^1} \|\nabla S_{k-1}v\|_{L^\infty}
  \|\nabla \Delta_k f\|_{L^2} \\
  & \lesssim \sum_{|k-q|\leq 4} 2^{q(s-\beta-1)}2^{k(1-s )}
  \sum_{-\infty< k_1\leq k-2}2^{k_1 \beta} 2^{k_1(1+\frac{n}{2}-\beta)} \|\dot\Delta_{k_1}v\|_{L^2}
  \big(2^{ks}\|\Delta_k f\|_{L^2}\big) \\
  &\lesssim  \|f\|_{B_{2,2}^s} \sum_{-\infty < k_1\leq q+2} 2^{(k_1-q) \beta} 2^{k_1(1+\frac{n}{2}-\beta)}
  \|\dot\Delta_{k_1}v\|_{L^2} \\
  & \lesssim c_q \| v\|_{\dot B_{2,2}^{1+\frac{n}{2}-\beta}} \|f\|_{B_{2,2}^s},
\end{aligned}
\end{equation*}
with $(c_q)_{q\in\mathbb{N}}$ satisfying $\sum_{q\in\mathbb{N}}(c_q)^2\leq 1$.
For $\mathrm{II}_q$, we directly obtain that for every $s<1+\frac{n}{2}$
\begin{equation*}
\begin{aligned}
  2^{q(s-\beta)}\|\mathrm{II}_q\|_{L^2}& \lesssim 2^{q(s-\beta)}\sum_{|k-q|\leq 4;k\in\mathbb{N}}
  \|\Delta_k v\|_{L^2} \|\nabla S_{k-1} f\|_{L^\infty} \\
  & \lesssim \sum_{|k-q|\leq 4;k\in\mathbb{N}} 2^{k(s-\beta)} \|\Delta_k v\|_{L^2} \sum_{k_1\leq k-2}2^{k_1(1+\frac{n}{2}-s)}
  \big(2^{k_1s}\|\Delta_{k_1} f\|_{L^2}\big) \\
  & \lesssim  \| v\|_{\dot B_{2,2}^{1+\frac{n}{2}-\beta}} \sum_{k_1\leq q+2}2^{(k_1-q)(1+\frac{n}{2}-s)}
  \big(2^{k_1s}\|\Delta_{k_1} f\|_{L^2}\big) \\
  & \lesssim c_q \|v\|_{\dot B_{2,2}^{1+\frac{n}{2}-\beta}} \|f\|_{B_{2,2}^s}.
\end{aligned}
\end{equation*}
In particular, when $n=2$ and $v=\mathcal{R}^\perp f$, using the Calder\'on-Zygmund theorem we get
\begin{equation*}
\begin{aligned}
  2^{q(s-\beta)}\|\mathrm{II}_q\|_{L^2}& \lesssim 2^{q(s-\beta)}\sum_{|k-q|\leq 4;k\in\mathbb{N}}
  \|\Delta_k v\|_{L^2} \|\nabla S_{k-1} f\|_{L^\infty} \\
  & \lesssim \sum_{|k-q|\leq 4;k\in\mathbb{N}} 2^{ks}\|\Delta_k f\|_{L^2} 2^{-k\beta}\sum_{-\infty< k_1\leq k-2} 2^{k_1\beta}
  2^{k_1(1+\frac{n}{2}-\beta)}\|\dot \Delta_{k_1}f\|_{L^2} \\
  & \lesssim c_q \|f\|_{\dot B_{2,2}^{1+\frac{n}{2}-\beta}} \|f\|_{B_{2,2}^s}.
\end{aligned}
\end{equation*}
From the divergence-free property of $v$, we further decompose $\mathrm{III}_q$ as follows
\begin{equation*}
  \mathrm{III}_q =\sum_{k\geq q-3;k\in\mathbb{N}; i}[\partial_i\Delta_q,\Delta_k v_i]\widetilde{\Delta}_k f +
  [\Delta_q, \Delta_{-1}v]\cdot\nabla\widetilde{\Delta}_{-1}f\triangleq \mathrm{III}_q^1 +\mathrm{III}_q^2.
\end{equation*}
For $\mathrm{III}^1_q$, from direct computation we find that
\begin{equation*}
\begin{aligned}
  2^{q(s-\beta)}\|\mathrm{III}^1_q\|_{L^2} & \lesssim 2^{q(s-\beta)}\Big(\sum_{k\geq q-3;k\in\mathbb{N}; i}
  \|\partial_i\Delta_q(\Delta_k v_i\widetilde{\Delta}_k
  f)\|_{L^2}+ \sum_{|k-q|\leq 2;k\in\mathbb{N}}\|\Delta_k v\cdot\nabla \Delta_q\widetilde{\Delta}_k f\|_{L^2}\Big)\\
  & \lesssim 2^{q(s-\beta)} \Big(2^{q(1+\frac{n}{2})}\sum_{k\geq q-3;k\in\mathbb{N}} \|\Delta_k v\|_{L^2}
  \|\widetilde{\Delta}_k f\|_{L^2}
  + \sum_{|k-q|\leq 2;k\in\mathbb{N}} 2^{k\frac{n}{2}}\|\Delta_k v\|_{L^2} 2^q\|\Delta_q f\|_{L^2}\Big) \\
  & \lesssim \|f\|_{ B_{2,2}^s}\sum_{k\geq q-3;k\in\mathbb{N}}2^{(q-k)(s-\beta+1+\frac{n}{2})}
  2^{k(1+\frac{n}{2}-\beta)}\|\Delta_k v\|_{L^2}
  + \|v\|_{\dot B_{2,2}^{1+\frac{n}{2}-\beta}} 2^{qs}\|\Delta_q f\|_{L^2} \\
  & \lesssim c_q \|v\|_{\dot B_{2,2}^{1+\frac{n}{2}-\beta}} \|f\|_{ B_{2,2}^s}.
\end{aligned}
\end{equation*}
For $\mathrm{III}_q^2$, due to that $\mathrm{III}_q^2=0$ for all $q\geq 3$, and similarly as estimating $\mathrm{I}_q$ we obtain
\begin{equation*}
\begin{aligned}
  2^{q(s-\beta)}\|\mathrm{III}_q^2\|_{L^2}& \lesssim 1_{q\in \{0,1,2\}} \|x h_q\|_{L^1} \|\nabla \Delta_{-1}v\|_{L^\infty}
  \|\widetilde{\Delta}_{-1}f\|_{L^2} \\
  & \lesssim 1_{q\in \{0,1,2\}} \sum_{-\infty < k_1\leq 0} 2^{k_1\beta} 2^{k_1(1+\frac{n}{2}-\beta)}\|\dot \Delta_{k'}v\|_{L^2}
  \|f\|_{L^2} \\ & \lesssim 1_{q\in \{0,1,2\}}  \|v\|_{\dot B_{2,2}^{1+\frac{n}{2}-\beta}} \|f\|_{ B_{2,2}^s}.
\end{aligned}
\end{equation*}
Gathering the upper estimates leads to the expected results.
\end{proof}

We also treat some product estimates.
\begin{lemma}\label{lem product}
  Let $v$ be a smooth divergence-free vector field over $\mathbb{R}^n$ and $f$ be a smooth scalar function of $\mathbb{R}^n$.
Then,
\begin{enumerate}
 \item if $f$ satisfies $\mathrm{supp}\widehat{f}\subset \{\xi: |\xi|\leq R \}$, a positive absolute constant $C$ can be found
 such that for every $q\in\mathbb{N}$, $\beta\in]0,\frac{n}{2}[$ and
 $s>\beta-1-\frac{n}{2}$,
 \begin{align}\label{eq prod}
   2^{q(s-\beta)} \|\Delta_q(v\cdot\nabla f)\|_{L^2} \leq C R^{1+s-\beta} c_q \|v\|_{L^2} \|f\|_{L^\infty} + C c_q \|v\|_{B_{2,2}^s}
   \| f\|_{\dot B_{2,2}^{1+\frac{n}{2}-\beta}},
 \end{align}
 with $(c_q)_{q\in\mathbb{N}}$ satisfying $\sum_{q\in\mathbb{N}}(c_q)^2\leq 1$. Especially, if $n=2$ and $v=\mathcal{R}^\perp f$,
 we further have that for all $\beta,s$ satisfying $s+1-\beta>0$,
 \begin{align}\label{eq prod1}
   2^{q(s-\beta)} \|\Delta_q(v\cdot\nabla f)\|_{L^2} \leq C R^{1+s-\beta} c_q \|f\|_{L^2} \|f\|_{L^\infty},
 \end{align}
 with $(c_q)_{q\in\mathbb{N}}$ satisfying $\sum_{q\in\mathbb{N}}(c_q)^2\leq 1$.
 \item for every $q\in \mathbb{Z}^-\cup\{0\}$, $\beta\in]0,\frac{n}{2}[$, there exists a positive absolute constant $C$ such that
 \begin{align}\label{eq prod2}
   \|\dot\Delta_q (v\cdot\nabla f)\|_{L^2}\leq C 2^{q(1+\frac{n}{2}-\beta)}  \||D|^\beta v\|_{L^2} \|f\|_{L^2},
 \end{align}
 and
 \begin{align}\label{eq prod3}
   \|\dot\Delta_q (v\cdot\nabla f)\|_{L^2}\leq C 2^{q(1+\frac{n}{2}-\beta)}  \| v\|_{L^2} \||D|^\beta f\|_{L^2}.
 \end{align}
\end{enumerate}

\end{lemma}

\begin{proof}[Proof of Lemma \ref{lem product}]
(1) We first prove \eqref{eq prod}. Thanks to Bony's decomposition, we have
\begin{equation*}
\begin{aligned}
  \Delta_q(v\cdot\nabla f)& =\sum_{|k-q|\leq 4}\Delta_q(S_{k-1}v\cdot\nabla \Delta_k f)
  + \sum_{|k-q|\leq 4} \Delta_q(\Delta_k v\cdot\nabla
  S_{k-1}f) + \sum_{k\geq q-3} \nabla\cdot\Delta_q(\widetilde{\Delta}_k v \Delta_k f) \\
  & \triangleq I_q +II_q +III_q.
\end{aligned}
\end{equation*}
For $I_q$, from the support property of $\widehat{f}$, we have
\begin{equation*}
  2^{q(s-\beta)}\|I_q\|_{L^2} \lesssim 2^{q(s-\beta)} \sum_{|k-q|\leq 4; 2^k \lesssim R} \|S_{k-1}v\|_{L^2} 2^{k}\|\Delta_k f\|_{L^\infty}
  \lesssim R^{1+s-\beta} c_q \|v\|_{L^2} \|f\|_{L^\infty},
\end{equation*}
with $(c_q)_{q\in\mathbb{N}}$ satisfying $\sum_{q\in \mathbb{N}} (c_q)^2\leq 1$. For the other two terms,
in a similar and simpler way as the treatment of $\mathrm{II}_q$ and $\mathrm{III}_q$, we obtain that
\begin{align*}
  2^{q(s-\beta)}\|II_q+III_q\|_{L^2}\lesssim c_q \|v\|_{B_{2,2}^s}
  \|f\|_{\dot B_{2,2}^{1+\frac{n}{2}-\beta}}.
\end{align*}
Next we treat \eqref{eq prod1}.
Since $\mathrm{supp}\,\widehat{v\cdot \nabla f}\subset \{\xi: |\xi|\leq 2R\}$, we find
\begin{equation*}
\begin{aligned}
  2^{q(s-\beta)}\|\Delta_q(v\cdot\nabla f)\|_{L^2} & \lesssim 1_{\{q;\; 2^q\lesssim R\}}
  2^{q(s-\beta +1)} \|\Delta_q (v f)\|_{L^2} \\
  & \lesssim    1_{\{q;\; 2^q\lesssim R\}} R^{s-\beta +1} \|f\|_{L^2} \|f\|_{L^\infty},
\end{aligned}
\end{equation*}
and it clearly implies \eqref{eq prod1}.
\\
(2) We then prove \eqref{eq prod2}. We also have the decomposition
\begin{equation*}
\begin{aligned}
  \dot\Delta_q(v\cdot\nabla f)& =\sum_{|k-q|\leq 4}\dot\Delta_q(\dot S_{k-1}v\cdot\nabla \dot\Delta_k f) + \sum_{|k-q|\leq 4}
  \dot\Delta_q(\dot\Delta_k v\cdot\nabla
  \dot S_{k-1}f) + \sum_{k\geq q-3} \nabla\cdot\dot\Delta_q(\dot\Delta_k v \widetilde{\dot\Delta}_k f) \\
  & \triangleq \dot{I}_q +\dot{II}_q + \dot{III}_q.
\end{aligned}
\end{equation*}
For $\dot I_q$, we directly have
\begin{equation*}
\begin{aligned}
  \|\dot I_q\|_{L^2} & \lesssim  \sum_{|k-q|\leq 4} \|\dot S_{k-1} v\|_{L^\infty}  2^k \|\dot \Delta_k f\|_{L^2} \\
  & \lesssim\sum_{|k-q|\leq 4} \sum_{-\infty< k_1\leq k-2} 2^{k_1(\frac{n}{2}-\beta)} 2^{k_1\beta}\|\dot \Delta_{k_1}v\|_{L^2}
   2^k \|\dot \Delta_k f\|_{L^2} \\
  & \lesssim 2^{q(1+\frac{n}{2}-\beta)} \||D|^\beta v\|_{L^2} \|f\|_{L^2}
\end{aligned}
\end{equation*}
For $\dot{II}_q$, from Bernstein's inequality we similarly get
\begin{align*}
  \|\dot{II}_q\|_{L^2} \lesssim \sum_{|k-q|\leq 4} 2^{k\frac{n}{2}}\|\dot\Delta_k v\|_{L^2}
  2^k \|\dot S_{k-1}f\|_{L^2}  \lesssim 2^{q(1+\frac{n}{2}-\beta)} \||D|^\beta v\|_{L^2} \|f\|_{L^2}.
\end{align*}
We treat $\dot{III}_q$ as follows,
\begin{equation*}
\begin{aligned}
  \|\dot{III}_q\|_{L^2} & \lesssim  \sum_{k\geq q-3} 2^{q(1+\frac{n}{2})} \|\dot\Delta_k v\|_{L^2}
  \|\widetilde{\dot\Delta}_k f\|_{L^2} \\
  & \lesssim 2^{q(1+\frac{n}{2})} \sum_{k \geq q-3} 2^{-k\beta} 2^{k\beta}\|\dot \Delta_k v\|_{L^2} \|f\|_{L^2} \\
  & \lesssim 2^{q(1+\frac{n}{2}-\beta)} \||D|^\beta v\|_{L^2} \|f\|_{L^2}.
\end{aligned}
\end{equation*}
Collecting the upper estimates yields \eqref{eq prod2}.
The proof of \eqref{eq prod3} is almost identical to the above process, and we omit it.

\end{proof}

{\bf Acknowledgement: }  
C. Miao and L.Xue were partly supported by the NSF of China
(No.11171033).


\begin{thebibliography}{60}

\bibitem{BMN96}A. Babin, A. Mahalov and B. Nicolaenko, Global splitting, integrability and regularity of 3D Euler and Navier-Stokes equations
             for uniformly rotating fluids. European Journal of Mechanics, \textbf{15} (1996), 291-300.
\bibitem{BMN99}A. Babin, A. Mahalov and B. Nicolaenko, Global regularity of 3D rotating Navier-Stokes equations for resonant domains.
           Indiana University Mathematics Journal, \textbf{48} (1999), 1133-1176.
\bibitem{BCD11}H. Bahouri, J.-Y. Chemin and R. Danchin, \textit{Fourier Analysis and Nonlinear Partial Differential Equations}, Grundlehren der
mathematischen Wissenschaften \textbf{343}, Springer-Verlag, (2011).

\bibitem{CV}L. Caffarelli and  V. Vasseur, Drift diffusion equations with fractional diffusion and the
                    quasi-geostrophic equations. Annals of Math. \textbf{171} Issue 3(2010), 1903-1930.
\bibitem{CDGG00}J.-Y. Chemin, B. Desjardins, I. Gallagher, and E. Grenier, Fluids with anisotropic viscosity.
           Mathematical Modelling and Numerical Analysis, \textbf{34}, no. 2 (2000), 315-335.
\bibitem{CDGG06}J.-Y. Chemin, B. Desjardins, I. Gallagher and E. Grenier, \textit{Mathematical geophysics: an introduction
              to rotating fluids and the Navier-Stokes equations}. Clarendon Press, Oxford, (2006).
\bibitem{ChenMZ}Q. Chen, C. Miao and Z. Zhang, A new Bernstein's inequality
                and the 2D dissipative quasi-geostrophic equation. Comm. Math. Phys., \textbf{271} (2007), 821-838.
\bibitem{CMT} P. Constantin, A.J. Majda and E. Tabak, Formation of strong fronts in the $2$-D quasigeostrophic thermal active
               scalar. Nonlinearity, \textbf{7}(1994), 1495-1533.


\bibitem{ConW}P. Constantin and J. Wu, Regularity of H\"older continuous solutions of the supercritical quasi-geostrophic equation,
          Ann. Inst. H. Poincar\'e Anal. Non Lin\'eaire, \textbf{25}(2008), No.6, 1103-1110.
\bibitem{AC-DC}A. C\'ordoba  and  D. C\'ordoba,  A maximum principle applied to the quasi-geostrophic equations.
               Comm. Math. Phys., \textbf{249}(2004), 511-528.
\bibitem{Dab}M. Dabkowski, Eventual regularity of the solutions to the supercritical dissipative quasi-geostrophic equation,
              Geom. Funct. Anal., \textbf{21} (2011), no. 1, 1-13.
\bibitem{Dong2}H. Dong, Dissipative quasi-geostrophic equations in critical Sobolev spaces: smoothing effect and global well-posedness.
                  Discrete Contin. Dyn. Syst., \textbf{26} Issue 4(2010), 1197-1211.
\bibitem{Gal98}I. Gallagher, Applications of Schochet's methods to parabolic equation.
                Journal de Math\'ematiques Pures et Appliqu\'ees, \textbf{77} (1998), 989-1054.
\bibitem{GalSR07}I. Gallagher and L. Saint-Raymond, On the influence of the Earth's rotation on geophysical flows,
          \textit{Handbook of Mathematical Fluid Dynamics}, S. Friedlander and D. Serre Editors Vol 4, Chapter 5, 201-329, 2007.
\bibitem{HeldPGS}I. Held, R. Pierrehumbert, S. Garner and K Swanson, Surface quasi-geostrophic dynamics. J. Fluid Mech., \textbf{282}(1995), 1-20.
\bibitem{HmiK}T. Hmidi and S. Keraani, Global solutions of the supercritical 2D dissipative quasi-geostrophic equation,
              \textit{Adv. Math.} \textbf{214}(2007), 618-638.
\bibitem{Ju05}N. Ju, Existence and uniqueness of the solution to the dissipative 2D quasi-geostrophic equations in the Sobolev space.
              Comm. Math. Phy., \textbf{251} (2004), 365-376.

\bibitem{KisNV}A. Kiselev, F. Nazarov and A. Volberg, Global well-posedness
         for the critical 2D dissipative quasi-geostrophic equation, Invent. Math. \textbf{167}(2007), 445-453.
\bibitem{KisN}A. Kiselev and F. Nazarov, Global regularity for the critical dispersive dissipative surface Quasi-Geostrophic equation,
               Nonlinearity, \textbf{23}(2010), 549--554.

\bibitem{Kis}A. Kiselev, Nonlocal maximum principle for active scalars. Adv. in Math., \textbf{227} no. 5(2011), 1806-1826.


\bibitem{Ngo09}V. Ngo, Rotating fluids with small viscosity. Int. Mat. Res. Not., \textbf{2009}(2009), no.10, 1860-1890.
\bibitem{Res}S. Resnick, Dynamical problems in nonlinear advective partial differential equations,
            Ph.D. thesis, University of Chicago, 1995.

\end{thebibliography}
\end{document}